\documentclass[letterpaper, 11pt,  reqno]{amsart}

\usepackage{comment}
\usepackage{bbm}
\usepackage{amsmath,amssymb,amscd,amsthm,amsxtra, esint}
\usepackage{amsfonts,latexsym}
\usepackage{eucal}
\usepackage{mathabx}

\usepackage[margin=1.2in,marginparwidth=1.5cm, marginparsep=0.5cm]{geometry}

\usepackage[implicit=true]{hyperref}
\allowdisplaybreaks[2]

\usepackage{color}

\definecolor{gr}{rgb}   {0.,   0.69,   0.23 }
\definecolor{bl}{rgb}   {0.,   0.5,   1. }
\definecolor{mg}{rgb}   {0.85,  0.,    0.85}
\definecolor{yl}{rgb}   {0.8,  0.7,   0.}
\definecolor{or}{rgb}  {0.7,0.2,0.2}

\newtheorem{theorem}{Theorem} [section]

\DeclareMathOperator*{\supp}{supp}

\newcommand{\Z}{\mathbb{Z}}
\newcommand{\R}{\mathbb{R}}

\newcommand{\N}{\mathbb{N}}

\newcommand{\nocontentsline}[3]{}
\newcommand{\tocless}[2]{\bgroup\let\addcontentsline=\nocontentsline#1{#2}\egroup}

\newtheorem*{ackno}{Acknowledgments}

\numberwithin{equation}{section}
\numberwithin{theorem}{section}

\newtheorem{thm}{Theorem}[section]

\newtheorem{lem}[thm]{Lemma}
\newtheorem{prop}[thm]{Proposition}

\newcommand{\abrac}[1]{\left\langle #1 \right\rangle}
\newtheorem{remark}[theorem]{Remark}
\newtheorem{definition}[theorem]{Definition}
\theoremstyle{definition}

\theoremstyle{remark}

\def\R{{\mathbb R}}

\def\N{{\mathbb N}}

\def\supp{\mathop{\rm supp}\nolimits}

\def\adnorm#1{\left\|#1\right\|}

\newcommand{\del}{\partial}

\renewcommand{\hat}{\widehat}
\renewcommand{\Tilde}{\widetilde}

\numberwithin{equation}{section}

\begin{document}
\title[5th Order KP]{Unconditional Uniqueness of 5th Order KP Equations}

\author{James Patterson}
\begin{abstract}
In this paper we study the $5$th Order Kadomstev-Petviashvili (KP) equations posed on the real line. In particular we adapt the energy estimate argument from Guo-Molinet \cite{guo2024well} to conclude unconditional uniqueness of the solution to data map for $5$th order KP type equations. Applying short-time $X^{s,b}$ methods to improve classical energy estimates  provides more than sufficient decay when considering estimates on the interior of the time interval $[0,T]$.
The issue is how we deal with the boundary. By abusing symmetry we can apply multilinear interpolation to gain access to $L^4$ Strichartz estimates, which provide improved derivative gain. When taken together, the regularity of our resultant function space can be arbitrarily close to $L^2$, which in the context of unconditional uniqueness results is almost sharp.
\end{abstract}
\maketitle
\tableofcontents
\section{Introduction}
Our work focuses on the Cauchy problem for $5$th Order KP type equations; for $\delta = \pm 1$
\begin{equation}\label{KPEquation}
\begin{split} 
&\del_x ( \del_t u + \del_x^5 u + \del_x(u^2) ) +\delta \del_y^2 u = 0 \\
&u(0) = u_0.
\end{split}
\end{equation}
where $u(t,x,y) : \R^3 \rightarrow \R$ is unknown with initial data $u_0 : \R^2 \rightarrow \R$. When $\delta = 1$ this is known as the $5$th Order KP-I while $\delta = -1$ is referred to as $5$th Order KP-II. In this paper, we aim to demonstrate the Unconditional Uniqueness of solutions to \eqref{KPEquation} to almost critical regularity. KP equations of $3$rd order were first introduced by Kadomstev and Petviashvili to model two dimensional water waves \cite{KadPet70} (also see \cite{ablowitz1979evolution}) as a natural extension of the well-known one dimensional Korteweg-de Vries (KdV) equation
\begin{equation} \label{KdV}
\del_t u + \del_x^3 u + \del_x(u^2) =0.
\end{equation}
Between KP-I and KP-II, KP-I aims to model water waves with strong surface tension, while KP-II covers the weak surface tension case. In the same way that generalised dispersion models of KdV has seen increasing study, so has generalised dispersions of KP type equations. Indeed starting from the following $5$th order version of the KdV
\begin{equation} \label{5KdV}
\del_t u + \del_x^5 u + \del_x(u^2) = 0,
\end{equation}
by adding weakly transverse perturbations to \eqref{5KdV} one can recover \eqref{KPEquation}; in the same way the usual $3$rd Order KP type equations arise from \eqref{KdV}. For real valued solutions to \eqref{KPEquation}, there are conserved quantities, mass and energy, of the form
\[ M[u](t) := \int_{\R^2} u^2 dx dy,\]
\[ E[u](t) := \int_{\R^2} \frac{1}{2} |\del_x^2 u |^2 + \frac{\delta}{2} |\del_x^{-1} \del_y u |^2 - \frac{1}{6}u^3 dx dy\]
respectively.

When comparing results between KP-I and KP-II, we can observe the value the defocusing present in KP-II plays. One way this can be observed is during analysis of each equations's resonance relation. If we take $\xi$ as the Fourier variable in the $x$ direction and $\mu$ in the $y$, for \eqref{KPEquation}, our resonance relation is
\begin{equation} \label{resonantRelation}
\xi^5 - \xi_1^5 - \xi_2^5 - \delta \frac{(\xi_1 \mu - \xi \mu_1)^2}{\xi \xi_1 \xi_2}.
\end{equation}
In the KP-II case, $\delta = -1$, both components of \eqref{resonantRelation} hold the same sign and hence compound, giving better estimates, while in the KP-I case they conflict. This behaviour is exhibited in the differences between the results of the KP-I and KP-II equations of all dispersions. The $3$rd Order KP-II equation initially had global well-posedness demonstrated by Bourgain \cite{bourgain1992cauchy} for appropriate initial data in $L^2$ in both periodic and non-periodic settings. This was further improved by Takaoka \cite{takaoka2000well} and then Takaoka-Tzvetkov \cite{takaoka2001local} with the use of Anisotropic Sobolev space; the sharp result in critical space was obtained by Hadac-Herr-Koch \cite{hadac2009well}. In comparison for the $3$rd Order KP-I, Molinet-Saut-Tzvetkov \cite{molinet2002well} showed the solution map cannot be $C^2$ smooth at the origin, hence a contraction mapping argument would inevitably fail. Despite this, \cite{molinet2007global}, using other methods, showed the local well-posedness, which has since been improved to $H^{1,0}$ by Guo-Peng-Wang \cite{guo2010local}.

For generalised dispersion KP type equations, Sanwal-Schippa \cite{sanwal2022low} displayed that with high enough dispersion, one can achieve $L^2$ global well-posedness for the KP-I type with either periodic or non-periodic spatial components; see also \cite{herr2024low} for analogous results in higher dimensions. Hadac \cite{hadac2008well} gave local well-posedness for generalised KP-II for dispersion lower than $3$rd order.

Most of the results discovered in the study of the KP equations are within Anisotropic Sobolev space $H^{s_1,s_2}$;
\begin{equation} \label{Anisotropic}
\| u \|_{H^{s_1,s_2}} := \left( \int_{\R^2} |\abrac{\xi}^{s_1}\abrac{\mu}^{s_2} \hat{u}(\xi,\mu)|^2 d\xi d\mu \right)^\frac{1}{2}.
\end{equation}
Once could foresee this being the case, not only considering \eqref{resonantRelation}, but also with the knowledge of the available Strichartz estimates. In general, given $U(t)$ as the linear propagator for a KP type equation, these take the form
\[ \|D_x^{-\gamma} U(t) \varphi \|_{L^q_T L^r_{xy}} \lesssim \|\varphi \|_{L^2_{xy}}, \]
exemplifying that the loss or gain of regularity only appears in the $x$ direction. For most results, this remains consistent. Hence it is usually taken $s_2 = 0$ in \eqref{Anisotropic}. 

Our aim in this paper, is to show an Unconditional Uniqueness result. For the results stated above, the uniqueness of solutions is shown within subspaces of $C_T H^{s_1,s_2}$, a prominent example being the anisotropic adapted Bourgain space $X^{s_1,s_2,b}$. This poses the question whether we can guarantee the uniqueness of the data to solution map for the full space $C_T H^{s_1,s_2}$. T.Kato in \cite{kato1995nonlinear} was the first to pose and prove results of this kind, which are important when considering the convergence of numerical solutions. Quite a few methods have been developed to show results of this type, a successful one including normal form reduction; see \cite{kishimoto2019unconditional,kwon2012unconditional,kwon2020normal, oh2021normal} for examples. This method makes strong use of the respective equation's resonance relation. However, since it is not clear from \eqref{resonantRelation} how one could recover $y$ directional derivative gain, the method seems ill-suited if demanding $s_2 = 0$ in \eqref{Anisotropic}. Instead our work is greatly inspired by Guo-Molinet \cite{guo2024well}, who used a combination of energy and short time $X^{s,b}$ methods to demonstrate the Unconditional Uniqueness of the $3$rd Order KP type equations, within $C_T H^{s,0}$ for $s > 3/4$. In comparison we show the following result.
\begin{thm} \label{mainTheorem}
Both $5$th Order KP-I and KP-II equations \eqref{KPEquation} admit unconditionally unique solutions within anisotropic Sobolev space $C_T H^{s,0}$ for $s > 0$.
\end{thm}
We emphasize that Theorem \ref{mainTheorem} is almost sharp. In the context of unconditional uniqueness results, we desire the nonlinearity to make sense as a distribution. In this scenario, we require $u^2 \in L^1_{\text{loc}}$ or rather $u \in L^2_{\text{loc}}$. This means Theorem \ref{mainTheorem} could only ever be improved at best to the critical case of
unconditional uniqueness within $C_T L^2$. The author is currently considering this result in a future work.

The method adapted from \cite{guo2024well}, uses classical energy estimates on solutions, similar to those seen in \cite{molinet2015improvement}, using a smoothing trilinear estimate. Found in \cite{ionescu2007global, sanwal2022low}, these trilinear estimates introduce both Bourgain spaces as intermediaries and derivative gain. This is further improved within \cite{guo2024well} using short-time $X^{s,b}$ methods. By decomposing the time interval $[0,T]$ into lengths $T/N$, dependent on the largest spatial frequency $N$, one can spread derivative loss between linear and integrand terms when the Duhamel formulation is substituted in. Strichartz estimates then allow us to gain bounds on $L^2_T L^\infty_{xy}$ norms which arise out of Hölder. Our main difference in argument is at the boundary of the time interval $[0,T]$. Since one needs to keep the domain of the time integral the same, one must use a sharp time cut-off at the boundary, which the short-time $X^{s,b}$ method cannot handle. In \cite{guo2024well}, they handle this with H\"older and $L^2_T L^\infty_{xy}$ Strichartz estimates; in our scenario, we can reduce the derivative loss with first a multi-linear interpolation, giving access to the use of $L^4_T L^4_{xy}$ Strichartz estimates.

From now on, unless prefixed otherwise, KP-I or KP-II will always refer to their $5$th order version.

\section{Preliminary}
We begin by defining recurring notation. For $x,y \in \R$, we write $x \lesssim y$ to denote that there exists constant $C > 0$ such that $x \leq Cy$; in the case that $0 < C < 1$ we additionally write that $x \ll y$. When both $x \lesssim y$ and $y \lesssim x$ we further write that $x \sim y$. When working with integral exponents $p \in [1, \infty]$ we will refer to $p'$ as the corresponding conjugate exponent, that is $\frac{1}{p} + \frac{1}{p'} = 1$. It will be necessary to sum over dyadic numbers, when doing so we will always capitalize, for example later on we will have $N \in 2^\Z$ or $L \in 2^\N$ etc.
Lastly, when $\varepsilon > 0$ can be taken arbitrarily small, we will often write $x+$ as shorthand for $x+ \varepsilon$; similarly for $x-$, and where $\infty-$ will refer to some arbitrarily large but finite value; in the context of integral exponents $\frac{1}{\varepsilon} + \frac{1}{\infty-} = 1$.
\subsection{Function spaces and Operator definitions}
Since we primarily work in Sobolev spaces, with use of Littlewood-Paley theory, it will of course be necessary to split frequency domains with the use of frequency projections. We make note of the conventions we will follow. Let $\kappa : \R \rightarrow [0,1]$ be a smooth radial cutoff for which $\kappa(\xi) \equiv 1$ for $|\xi| \leq 5/4$ and $\supp \kappa \subset [-8/5, 8/5]$. We then define a smooth annulus cutoff $\phi: \R_+ \rightarrow [0,1]$ by $\phi(\xi) := \kappa(|\xi|/2) - \kappa(|\xi|)$, which we extend to higher scales with $\phi_N(\xi) := \phi(\xi/N)$; in this way $\phi_N$ is supported on $|\xi| \sim N$. 

For the Fourier transform of $f \in \mathcal{S}'$ we follow the usual convention to write $\hat{f}$ (or $\mathcal{F}[f]$ when more appropriate) to denote the Fourier transform in both spatial and temporal variables:
\[ \hat{f}(\xi, \mu, \tau) := \int_{\R^3} e^{-ix\xi}e^{-iy\mu}e^{-it\tau}f(x,y,t) \, dxdydt.\]
We likewise write $\widecheck{f}$ or $\mathcal{F}^{-1}[f]$ to denote the inverse. Often we apply the Fourier transform only in the $x$ component; when doing so we will only make reference to the variable $\xi$ or use $\mathcal{F}_x$ instead. 
We define $P_N f$ as the Fourier multiplier $\hat{P_N f}(\xi) = \phi_N(\xi) \hat{f}(\xi)$, the projection of frequencies down to  frequencies $|\xi| \sim N$. Naturally we continue by defining
\[P_{\lesssim N}f = \sum_{M \lesssim N} P_M f\]
with $P_{\ll N}$ defined in the obvious analogous way. $\Tilde{P}_N$ is used to denote a projection to a larger frequency support, specifically $\Tilde{P}_N f := P_{N/2}f + P_N f + P_{2N} f$, for which it is quick to note $P_N \Tilde{P}_N = P_N$ and we write $\Tilde{\phi}_N := \phi_{N/2} + \phi_N +\phi_{2N}$ for the corresponding multiplier. For Theorem \ref{mainTheorem} we show results in the relevant Anisotropic Sobolev space:
\[ \adnorm{u}_{H^{s,0}} := \left(\int_{\R^2}| \abrac{\xi}^s \hat{u}(\xi,\mu) |^2 \, d\xi d\mu  \right)^\frac{1}{2} \sim \left(\sum_{N } \abrac{N}^{2s}\adnorm{P_N u}_{L^2_{xy}}^2 \right)^\frac{1}{2}.\]
For KP type equations
\[ \del_x(\del_t u + \del_x^5 u + u\del_x u) \pm \del_y^2 u = 0,\]
since we wish to exploit their corresponding dispersion $\tau - w_{\pm}(\xi,\mu) = \tau - \left(\xi^5 \pm \mu^2/\xi\right)$ where $w_+$ corresponds to KP-I while $w_-$ corresponds to KP-II, a famous way to incorporate is with the use of Bourgain spaces. To this end, we define the dispersion projection; for $L > 1$ we define the Fourier multiplier $Q_L$ such that 
\[\mathcal{F}[Q_L f](\xi,\mu,\tau) = \phi_L(\tau - w_{\pm}(\xi,\mu)) \hat{f}(\xi,\mu,\tau);\] 
with its use dispersion has been localized $|\tau - w_\pm(\xi,\mu)| \sim L$. For $L = 1$ we instead use
\[\mathcal{F}[Q_1f](\xi,\mu,\tau) = \sum_{0< L \leq 1}\phi_L(\tau - w_\pm(\xi,\mu)) \hat{f}(\xi,\mu,\tau). \]
Besov Bourgain space is thus defined as
\begin{equation} \label{BourgainNorm}
\|u\|_{X^{0,1/2,1}} := \sum_{L\geq 1} L^\frac{1}{2}\|Q_L u \|_{L^2_{xyt}}.
\end{equation}
Bourgain spaces are designed to quantify how far a function differs from the solution of the linear PDE; Bourgain spaces naturally have nice properties with relation to the linear propagator arising from the linear equation. Furthermore, Strichartz estimates provide a convenient way to convert $L^q$ norms into $L^2$ in exchange for derivative loss/gain; when combined with the transference principle, Lemma \ref{Transference} below, it also provides a means for which to introduce the Bourgain norm \eqref{BourgainNorm}. As such we make note of the linear propagator for KP equations here:
\[ \mathcal{F}[U(t)f](\xi,\mu) = e^{itw_\pm(\xi,\mu)}\hat{f}(\xi,\mu).\]
Since the resulting Strichartz estimates for KP-I and KP-II will be exactly the same, we will not make an effort to distinguish their linear propagators.

\subsection{Well Known Lemmas}
We quickly list off some essential lemmas which will find use later. 
The following is essential in introducing the Besov Bourgain space as an intermediary. See Lemma 2.9 in \cite{tao2006nonlinear}.
\begin{lem}[Transference Principle]
\label{Transference}
Let $\mathcal{T}$ be a $k$-linear operator such that for $1 \leq q,r \leq \infty$,
\[\|\mathcal{T}(U(t)f_1, \dots, U(t)f_k)\|_{L^q_t L^r_x} \lesssim \prod_{j=1}^k \|f_j\|_{L^2}\]
for all $f_j \in L^2_x$. Then
\[\|\mathcal{T}(u_1,\dots,u_k)\|_{L^q_t L^r_x} \lesssim \prod_{j=1}^k \|u_j\|_{X^{0,1/2,1}}\]
for all $u_j \in X^{0,1/2,1}$.
\end{lem}
The Strichartz estimates later, of course, satisfy the hypothesis of the transference principle, and will be the avenue in which a Besov Bourgain norm is imparted, especially in the KP-II case.

The dispersion present in Bourgain spaces gives rise to what we refer to as the resonance relation; for KP with quadratic nonlinearity, the relation is equal to $w_\pm(\xi,\mu) - w_\pm(\xi_1, \mu_1) - w_\pm(\xi - \xi_1, \mu - \mu_1)$. The natural benefit of Bourgain spaces as an intermediary is the ability of using lower bounds on the resonance relation to offset the derivative in the nonlinearity. The following gives bounds on only a component of the resonance relation, but in the case of KP-II this will turn out to be sufficient. See appendix A for proof.
\begin{lem}[Resonance Relation Estimate]\label{resonance}
For $\omega(\xi) = \xi^5$. With
\[ \Omega(\xi,\xi_1) := \omega(\xi) - \omega(\xi_1) - \omega(\xi - \xi_1), \hspace{0.5cm} \xi_1, \xi \in \R\]
we have
\[ \frac{4}{2^4}|\xi_\text{min}||\xi_\text{max}|^4 \leq |\Omega(\xi,\xi_1)| \leq (5 +\frac{1}{2^4})|\xi_\text{min}||\xi_\text{max}|^4; \]
where $|\xi_\text{min}| = \min\{|\xi|, |\xi_1|, |\xi-\xi_1| \}
$ and $|\xi_\text{max}| = \max\{|\xi|, |\xi_1|, |\xi-\xi_1| \}
$. 
\end{lem}

Next we state a specific case of the following product estimate; see Corollary 1.1 of \cite{oh20201} and references within.
\begin{lem}[Fractional Product Rule]
\label{fractionalProd}
Let $1 \leq p,q \leq \infty$, $\frac{1}{2} \leq r \leq \infty$ satisfy $\frac{1}{p} + \frac{1}{q} = \frac{1}{r}$. If $\theta > \max(0,\frac{1}{r} - 1)$ then
\begin{equation}
\| J_x^\theta(fg)\|_{L^r(\R)} \lesssim \|J_x^\theta f\|_{L^p(\R)} \|g\|_{L^q(\R)} + \| f\|_{L^p(\R)} \|J_x^\theta g\|_{L^q(\R)}.
\end{equation}
\end{lem}

 The following will be required to prove the Strichartz estimates; see \cite{stein1993harmonic} for proof.
\begin{lem}[Hardy-Littlewood-Sobolev Inequality] \label{HLSI}
For $0 < \alpha < n, 1 < p <  q< \infty$ and $\frac{1}{q} = \frac{1}{p} - \frac{\alpha}{n}$, the following holds
\[ \adnorm{\int_{\R^n} \frac{f(y)}{|x-y|^{n-\alpha}} dy}_{L^q(\R^n)} \lesssim \adnorm{f}_{L^p(\R^n)}. \]
\end{lem}
Finally, in a couple of scenarios, we will need to make use of symmetry to change the Lebesgue norm exponent. The following allows us to justify this change of norm. (See Theorem 4.4.1 in \cite{bergh2012interpolation} for more generality and proof)
\begin{lem}[Multilinear Interpolation] \label{Multiinterpolation}
For $n \in \N$, let $(q_0^1,q_0^2,\dots,q_0^n), (q_1^1,q_1^2,\dots,q_1^n) \in [1,\infty]^n$ and $X$ a Banach space. Suppose $T$ is a bounded multi-linear operator from $\prod_{i=1}^n L^{q_0^i} \rightarrow X$ and $\prod_{i=1}^n L^{q_1^i} \rightarrow X$. That is
\[ \|T(u_1,\dots,u_n)\|_X \leq M_0 \prod_{i=1}^n \|u_i\|_{L^{q_0^i}}\]
and
\[ \|T(u_1,\dots,u_n)\|_X \leq M_1 \prod_{i=1}^n \|u_i\|_{L^{q_1^i}}.\]
Then for $\theta \in [0,1]$
and $q_\theta^i$ defined as
\[ \frac{1}{q_\theta^i} = \frac{\theta}{q_0^i} + \frac{1-\theta}{q_1^i},\]
$T$ extends to a bounded multi-linear operator $\prod_{i=1}^n L^{q_\theta^i} \rightarrow X$ with
\[\|T(u_1,\dots,u_n)\|_X \leq M_0^\theta M_1^{1-\theta} \prod_{i=1}^n \|u_i\|_{L^{q_\theta^i}}.\]
\end{lem}

\section{Strichartz and Trilinear Estimates}
We aim to show a Strichartz estimate for the $5$th
order KP equations. This is essentially following the same proof for the $3$rd order KP case (see \cite{guo2024well}), though of course the resultant derivative loss ends up different.
\begin{definition}
We say a pair of exponents $(q,r)$ are \textbf{admissible} if they satisfy the following
\begin{itemize}
    \item $2 \leq q,r \leq \infty$
    \item $\frac{1}{2} \big( \frac{1}{2} - \frac{1}{r} \big) \leq \frac{1}{q} \leq \frac{1}{2} - \frac{1}{r}$
    \item $(q,r) \neq (2,\infty), (4, \infty)$
\end{itemize}
We define the corresponding constant $\beta(q,r) := \frac{4}{2} - \frac{4}{r} - \frac{5}{q}$.
\end{definition}

\begin{lem}[Strichartz Estimate]
Let $T > 0$, and assume that exponents $(q,r), (\Tilde{q}, \Tilde{r})$ are admissible. Then the following hold
\begin{equation}\label{Strichartz1}
    \adnorm{D_x^{-\beta(q,r)}U(t)\varphi}_{L^q_TL^r_{xy}} \lesssim \adnorm{\varphi}_{L^2_{xy}}
\end{equation}
and
\begin{equation}\label{Strichartz2}
\adnorm{\int_0^t D_x^{-\beta(q,r)-\beta(\Tilde{q}, \Tilde{r})} U(t-s)f(s) ds}_{L^q_T L^r_{xy}} \lesssim \adnorm{f}_{L^{\Tilde{q}'}_T L^{\Tilde{r}'}_{xy}}.
\end{equation}
\end{lem}
\begin{proof}
We begin by proving \eqref{Strichartz1}. Note $D_x^{-\beta(q,r)}U(t): L^2_{xy} \rightarrow L^q_TL^r_{xy}$ has adjoint operator $U^*: L^{q'}_TL^{r'}_{xy} \rightarrow L^2_{xy}$ with $U^*f := \int_0^T D_x^{-\beta(q,r)}U(-s)f(s) ds$. Hence, by the $TT^*$ method \eqref{Strichartz1} is equivalent to showing

\begin{equation} \label{TT*}
\adnorm{\int_0^t D_x^{-2\beta(q,r)}U(t-s)\varphi(s) ds}_{L^q_TL^r_{xy}} \lesssim \adnorm{\varphi}_{L^{q'}_TL^{r'}_{xy}}.
\end{equation}

Consider that we can rewrite, for $N \in 2^\Z$;
\[ U(t)P_N \varphi = G(\cdot,\cdot,t)*\varphi\]
where
\[ G(x,y,t) := \int_{\R^{2}} e^{i(x\xi + y \mu)} e^{it\left(\xi^5 \pm \frac{|\mu|^2}{\xi}\right)} \phi\left( \frac{\xi}{N}\right)d\xi d\mu.\]
\begin{remark}
As an aside, we can calculate the integral 
\[ \int_\R e^{iy\mu} e^{\pm it\frac{\mu^2}{\xi}} d\mu\]
by making the substitution $\mu = \sqrt{\frac{\xi}{t}}\eta$ leading to
\[ \int_\R t^{-\frac{1}{2}}|\xi|^\frac{1}{2} e^{i\left(y\sqrt{\frac{\xi}{t}}\eta \pm \eta^2 \right)} d \eta. \]
Further substituting $\eta =\zeta \mp \frac{y}{2}\sqrt{\frac{\xi}{t}}$ leads to
\[  t^{-\frac{1}{2}}\xi^\frac{1}{2}e^{\pm \frac{iy^2\xi}{4t}} \int_\R e^{\mp i\zeta^2} d\zeta = t^{-\frac{1}{2}}\xi^\frac{1}{2}e^{\pm\frac{iy^2\xi}{4t}}\sqrt{\frac{\pi}{2}}(1\pm i). \]
\end{remark}
By applying the above, we observe that
\begin{align*}
|G(x,y,t)| &\leq \sqrt{2\pi}\left|\int_\R t^{-\frac{1}{2}}|\xi|^{\frac{1}{2}} e^{ix\xi + it\xi^5}  e^{\pm \frac{iy^2\xi}{4t}}\phi\left(\frac{\xi}{N} \right) d\xi\right| \\
& \lesssim  t^{-\frac{1}{2}}\sup_{x \in \R} \left|\int_\R |\xi|^{\frac{1}{2}} e^{ix\xi + it\xi^5} \phi\left(\frac{\xi}{N} \right) d\xi\right|.
\end{align*}
We of course have the obvious estimate 
\begin{equation} \label{G1Est}
|G(x,y,t)| \leq t^{-\frac{1}{2}} \int_\R |\xi|^{\frac{1}{2}} \phi\left( \frac{\xi}{N} \right) d\xi \lesssim t^{-\frac{1}{2}} N^{\frac{3}{2}} .
\end{equation}
Also by substitution $\eta = \xi/N$;
\[ \sup_{x \in \R} \left|\int_\R |\xi|^{\frac{1}{2}} e^{ix\xi + it\xi^5} \phi\left(\frac{\xi}{N} \right) d\xi\right| = N^{\frac{1}{2}}N\sup_{x \in \R} \left|\int_\R |\eta|^{\frac{1}{2}} e^{ixN\eta + itN^{5}\eta^5} \phi\left(\eta\right) d\eta\right|.\]
Taking then $g(\eta) = |\eta|^\frac{1}{2}\psi(\eta)$ and $h(\eta) = xN\eta + tN^{5} \eta^5$, we have an Oscillatory integral, with which an application of Van der Corput's Lemma (see \cite{guo2008decay, stein1993harmonic}) yields
\[ \left|\int_\R e^{ih(\eta)} g(\eta) d\eta \right| \lesssim t^{-\frac{1}{2}}N^{-\frac{5}{2}}.\]
Thus in addition to \eqref{G1Est}, we also have
\begin{equation} \label{GEst2}
|G(x,y,t)| \lesssim t^{-1}N^{-1}.
\end{equation}
Interpolating between \eqref{G1Est} and \eqref{GEst2}, we have for $\theta \in [0,\frac{1}{2}]$,
\begin{equation}
|G(x,y,t)| \lesssim t^{-(\frac{1}{2}+ \theta)}N^{\frac{3}{2} -5\theta}.
\end{equation}
By Young's Convolution Inequality, it then follows that 
\begin{equation} \label{Linfty-L1}
\adnorm{U(t)P_N\varphi}_{L^{\infty}_{xy}} = \adnorm{G(\cdot, \cdot, t) * \varphi}_{L^\infty_{xy}} \leq \adnorm{G}_{L^\infty_{xy}}\adnorm{\varphi}_{L^1_{xy}} \lesssim  t^{-(\frac{1}{2} +\theta)}N^{\frac{3}{2}- 5\theta}\adnorm{\varphi}_{L^1_{xy}}.\end{equation}
Since $|e^{it\xi^5 \pm \frac{\mu^2}{\xi}}\phi_N(\xi)| \leq 1$ it is of course true that
\begin{equation} \label{L2-L2}
\adnorm{U(t)P_N\varphi}_{L^{2}_{xy}} \lesssim \adnorm{\varphi}_{L^2_{xy}}.
\end{equation}
Hence, we can interpolate between \eqref{Linfty-L1} and \eqref{L2-L2}, to get a general estimate for $L^r$ for $r \geq 2$:
\begin{equation} \label{GIntEst}
\adnorm{U(t)P_N\varphi}_{L^{r}_{xy}} \lesssim t^{-(\frac{1}{2}+\theta)(1-\frac{2}{r})}N^{(\frac{3}{2} - 5\theta)(1-\frac{2}{r})}\adnorm{\varphi}_{L^{r'}_{xy}}.
\end{equation}
We can now show a frequency truncated version of \eqref{TT*}. By Minkowski's Integral Inequality, and then \eqref{GIntEst}: 
\begin{align} \adnorm{\int_0^t U(t-s)P_N\varphi(s) ds}_{L^q_T L^r_{x,y}} &\leq \adnorm{\int_0^t \adnorm{U(t-s)P_N\varphi(s)}_{L^r_{xy}} ds}_{L^q_T} \nonumber \\ & \lesssim \adnorm{ \int_0^t |t-s|^{-(\frac{1}{2}+\theta)(1-\frac{2}{r})}N^{(\frac{3}{2}-5\theta)(1-\frac{2}{r})}\|\Tilde{P}_N\varphi \|_{L^{r'}_{xy}} }_{L^q_T} \nonumber\\
&\lesssim N^{2\beta(q,r)} \|\varphi \|_{L^{q'}_TL^{r'}_{xy}},
\end{align}
where we have applied Hardy-Littlewood-Sobolev Inequality (Lemma \ref{HLSI}), using the choice of $\theta = \frac{2}{q}\left(\frac{r}{r-2}\right)- \frac{1}{2}$, which leads to the above! Hence, by $TT^*$ method, by applying to $\Tilde{P}_N\varphi$ it follows
\begin{equation} \label{TruncatedStrichartz}
\adnorm{U(t)P_N\varphi}_{L^q_T L^{r}_{xy}} \lesssim N^{\beta(q,r)}\| \Tilde{P}_N\varphi \|_{L^{2}_{xy}}.
\end{equation}
For $r < \infty$, we can complete the proof by applying the Littlewood-Paley Square function estimate, Minkowski Integral inequality and finally \eqref{TruncatedStrichartz} to get
\begin{align*} \| D_x^{-\beta(q,r)}U(t)\varphi\|_{L^q_TL^r_{xy}} & \lesssim \| D_x^{-\beta(q,r)}U(t)P_N\varphi\|_{L^q_TL^r_{xy}\ell^2_N} \leq \| D_x^{-\beta(q,r)}U(t)P_N\varphi\| _{\ell^2_NL^q_TL^r_{xy}} \\
& \lesssim \|\Tilde{P}_N\varphi \|_{\ell^2_N L^2_{xy}} \lesssim \|\varphi\|_{L^2_{xy}}.
\end{align*}
Littlewood-Paley does not hold in $L^\infty$, so we are forced to use a different method for $r= \infty$; instead we use $(\theta,1)$-real interpolation between the following two estimates
\begin{equation}
\begin{split} \label{l1InterpolationEst}
\|U(t)P_N\varphi\|_{\ell^{-\beta(q+,r),2}_N L^{q+}_{T}L^{\infty}_{xy}}&\lesssim \|\varphi\|_{L^2} \\
\|U(t)P_N\varphi\|_{\ell^{-\beta(q-,r),2}_N L^{q-}_{T}L^{\infty}_{xy}}&\lesssim  \|\varphi\|_{L^2}.
\end{split}
\end{equation}
By Theorem 5.6.2 and Theorem 5.2.1  from \cite{bergh2012interpolation}, we get
\begin{equation}
(\ell^{-\beta(q+,r),2}L^{q+}_{T}L^{\infty}_{xy}, \ell^{{-\beta(q-,r)},2}L^{q-}_{T}L^{\infty}_{xy})_{\theta,1}=\ell^{-\beta(q,r),1}L^{q,1}_{T}L^{\infty}_{xy}.
\end{equation}
That is, Theorem 5.6.2, with $\beta(q,r) = \theta \beta(q+,r) + (1-\theta)\beta(q-,r)$ for some $\theta$, implies
\[ (\ell^{-\beta(q+,r),2}A_0, \ell^{-\beta(q-,r),2}A_1)_{\theta,1} =\ell^{-\beta(q,r),1}(A_0,A_1)_{\theta,1}\]
and Theorem 5.2.1 with $\frac{1}{q} = \frac{\theta}{q+} + \frac{1-\theta}{q-}$ for some $\theta$,
gives
\[ (L^{q+}_T,L^{q-}_T)_{\theta,1} = L^{q,1}_T.\]
Thus we can get 
\[
\|D_x^{-\beta(q,r)} U(t)\varphi\|_{L^{q}_{T}L^{\infty}_{xy}}\lesssim \|U(t)P_N\varphi \|_{\ell^{-\beta(q,r),1}_N L^{q}_{T}L^{\infty}_{xy}}\lesssim \|\varphi\|_{L^2}. 
\]
This completes the proof for \eqref{Strichartz1}.
We remark that one of \eqref{l1InterpolationEst} not holding in the case $q = 4$ or $q = 2$, is the reason why $(q,r) = (2,\infty), (4,\infty)$ are not admissible pairs.
To prove \eqref{Strichartz2}, apply \eqref{Strichartz1} with admissible pair $(q,r)$ and then its adjoint version which again follows from the $TT^*$ method, but instead with admissible pair $(\Tilde{q},\Tilde{r})$. 
\end{proof}

When finding \textit{a-priori} estimates, we work with the energy of the relevant KP equation. For general nonlinearity $f$, we consider the relevant $5$th order KP type equation, sometimes written as
\begin{equation} \label{KP5Nonlinef}
\del_t u + \del_x^5 u + \frac{1}{2}\del_x(f) \pm \del_x^{-1}\del_{yy}u = 0.
\end{equation}
After applying $P_N$ to \eqref{KP5Nonlinef}, multiply by $P_N u$ to obtain
\[ \partial_t P_N u \cdot P_N u + \partial_{x}^5 P_N u \cdot P_N u + \frac{1}{2} \partial_x P_N(f) \cdot P_N u \pm \partial_x^{-1}\partial_{yy} P_N u \cdot P_N u  = 0.\]
Integrating in the spatial variables, we observe the following terms come to nothing:
\[ \int \partial_{x}^5 P_N u\cdot P_N u = -\frac{1}{2}\int \del_x(\del_x^2 P_N u)^2 = 0\]
\[ \int \partial_x^{-1}\partial_{yy} P_N u \cdot P_N u = -\frac{1}{2}\int \partial_x (\partial_y \partial_x^{-1}P_N u)^2 = 0.\]
Hence, after integrating in time, solutions to \eqref{KP5Nonlinef} must satisfy
\begin{equation} \label{Energy}
\frac{1}{2}\adnorm{P_N u(t)}_{L^2}^2 - \frac{1}{2}\adnorm{P_N u(0)}_{L^2}^2 + \int_0^t \int_{\R^2}\partial_x(P_Nf )P_N u =0.
\end{equation}
In our applications, the nonlinearity is quadratic; it takes the form $f = u_1u_2$ for some functions $u_1,u_2$. We can split up the frequencies of $P_N(u_1u_2)$ into low-high and high-high interactions;
\begin{align*}
P_N(u_1u_2) &= P_N(P_{\ll N}u_1 \Tilde{P}_Nu_2) + P_N(\Tilde{P}_{ N}u_1 P_{\ll N}u_2) + \sum_{M \gtrsim N}P_N(P_Mu \Tilde{P}_Mu) \\
& = \sum_{K \ll N} P_N(P_{K}u_1 \Tilde{P}_Nu_2) +P_N(\Tilde{P}_{ N}u_1 P_{K}u_2) + \sum_{M \gtrsim N}P_N(P_Mu \Tilde{P}_Mu).
\end{align*}
Thus, when estimating the remaining term present in \eqref{Energy}
\begin{equation} \label{EnergyIntegrand}
\int_0^t \int_{\R^2}\partial_xP_N(u_1 u_2)P_N u,
\end{equation}
we observe that we can restrict the frequency interactions down to one case. Namely, for frequency projections $P_{N_1}u_1, P_{N_2}u_2$ and $P_N u$, then relabeling frequencies based on order $N_\text{max} \geq N_{\text{med}} \geq N_{\text{min}}$, we need only consider $N_\text{max} \sim N_{\text{med}} \gtrsim N_{\text{min}}$.

Since we are left to find estimates on terms of the form \eqref{EnergyIntegrand}, the following helps categorise the different instances in which it arises.

\begin{definition}
For $u,v \in L^2(\R)$ and $a \in L^\infty(\R^2)$, we define $\Lambda_a$ as the \textbf{convolution multiplier} with symbol $a$, specifically
\[ \mathcal{F}_x[\Lambda_a(u,v)](\xi) := \int_{\xi_1+\xi_2 = \xi} a(\xi_1, \xi_2) \hat{u}(\xi_1)\hat{v}(\xi_2).\]
We say that $a$ is \textbf{acceptable} if 
\begin{equation} \label{acceptable} \|\Lambda_a(u,v)\|_{L^2_x} \lesssim \|u\|_{L^\infty_x} \|v\|_{L^2_x}.
\end{equation}

Furthermore, we define for $u_i \in L^2(\R^{3})$ and $a \in L^\infty(\R^{2})$

\[ \Gamma_a(u_1,u_2,u_3):= \int_{\sum_{i=1}^3(\tau_i,\xi_i,\mu_i) =0} a(\xi_2,\xi_3) \prod_{i=1}^3 \hat{u}_i(\tau_i,\xi_i,\mu_i).\]
One can see that this leads to the expression
\begin{equation}\label{Gamma} \Gamma_a(u_1,u_2,u_3) = \iiint u_1(x,y,t) \Lambda_a(u_2(\cdot,y,t), u_3(\cdot,y,t))(x) dt dx dy.\end{equation}
\end{definition}
An important aspect about $\Gamma_a$ is that there is an in-built symmetry,
that is
\begin{equation} \label{roleswitching}
\Gamma_a(u_1,u_2,u_3) = \Gamma_{\Tilde{a}}(u_3,u_1,u_2) = \Gamma_{\Tilde{\Tilde{a}}}(u_2,u_3,u_1)
\end{equation}
where $\Tilde{a}(\xi_1,\xi_2) = a(-\xi_1 - \xi_2, \xi_1)$ and $\Tilde{\Tilde{a}}(\xi_1,\xi_2) = a(\xi_2, -\xi_1 - \xi_2)$. This is an important aspect; it allows for interpolation on bounds for $\Gamma_a$, but requires the knowledge that $a, \tilde{a}$ and $\tilde{\tilde{a}}$ are all acceptable with respect to \eqref{acceptable} which does not always necessarily follow.

We briefly consider the following non-trivial example; it is requirement that this is acceptable, as it allows for shifting around of derivatives from high to low frequencies. 
\begin{lem} \label{Commutator}
For $g \in L^\infty(\R)$, $h \in L^2(\R)$ and $ N_3\ll N$, the following commutator estimate holds
$$
\|\Lambda_{a_1}(g,h)\|_{L^2_x}=N_3^{-1}  \|[\partial_x P_N, P_{N_3}g] \tilde{P}_N h\|_{L^2_x}\lesssim \|g\|_{L^\infty_x} \|h\|_{L^2_x}
$$
where 
\begin{equation}\label{a1}
a_1(\xi_1,\xi_2)=N_3^{-1}\phi_{N_3}(\xi_1)\tilde \phi_N(\xi_2)[\phi_N(\xi_1+\xi_2)(\xi_1+\xi_2)-\phi_N(\xi_2)\xi_2] \;,
\end{equation}
with $a_1 \in L^\infty(\R^2)$
\end{lem}
Proof can be found in Lemma 3.3 from \cite{koch2003local} but included in the appendix for completeness.
We now prove an important trilinear estimate, this will be in essence the main way of estimating the trilinear terms of the form of $\Gamma_a$ mentioned before. We should note, only the condition that $a \in L^\infty$ is required and not full acceptability with respect to \eqref{acceptable}.
\begin{prop}[Trilinear Estimate]\label{Pro1}
Let $ a\in L^\infty(\R^2) $ with $ \|a\|_{L^\infty} \lesssim 1 $ be acceptable.
For any triple $(N_1,N_2,N_3) \in (2^{\Z})^3$ with $N_1\sim N_2 \gtrsim  N_3>0 $ and any $ u_1, \, u_2,\, u_3\in X^{0,1/2,1}$, it holds 
\begin{equation}\label{triestX1}
\Bigl| \Gamma_a \Bigl(  P_{N_1} u_1, P_{N_2}  u_2 , P_{N_3}  u_3 \Bigr) \Bigr| 
\lesssim N_1^{ - \frac{9}{4}}N_3^{-\frac{3}{4}} \prod_{i=1}^3 \|P_{N_i} u_i \|_{X^{0,1/2,1}} \; .
\end{equation}
\end{prop}
\begin{proof}
Note that we need only show this estimate for $Q_{L_i}u_i$, for any triple $(L_1,L_2,L_3) \in (2^\N)^3$, as we may simply take out the sum in $L$, from $\sum_L Q_L \equiv 1$, to get the sum over $L$ required in the Besov type Bourgain space; that is we need only show 
\begin{equation}
A := \Bigl| \Gamma_a \Bigl(  Q_{L_1}P_{N_1} u_1, Q_{L_2}P_{N_2}  u_2 , Q_{L_3}P_{N_3}  u_3 \Bigr) \Bigr|
\lesssim N_1^{-\frac{9}{4}}N_3^{-\frac{3}{4}}   \prod_{i=1}^3 \|Q_{L_i}P_{N_i} u_i \|_{X^{0,1/2,1}} .
\end{equation}

For the KP-II case, we may also assume $\max(L_1,L_2,L_3) \ge 2^{-8} N_1^{4} N_3 $. By the definition of $Q_L$, we assume that $|\omega_-(\tau_i, \xi_i, \mu_i)| \sim L_i$, and by the resonance relation:
\[ \omega_-(\xi_1,\mu_1,\tau_1) + \omega_-(\tau_2, \xi_2, \mu_2) - \omega_-(\tau,\xi,\mu) = \xi^5 - \xi_1^5 - \xi_2^5 +\frac{(\mu\xi_1 - \mu_1\xi)^2}{\xi_1\xi_2\xi} , \]
having used that $\sum_{i=1}^3(\tau_i,\xi_i,\mu_i) =0$. Hence by Lemma \ref{resonance}
\[ \max(L_1,L_2,L_3) \gtrsim |\omega_{-}(\xi_1,\mu_1,\tau_1) + \omega_{-}(\tau_2, \xi_2, \mu_2) - \omega_{-}(\tau,\xi,\mu)| \gtrsim N_1^4 N_3.\]
Using symmetry, without loss of generality, we may assume that $L_1$ is the maximum. We make use of Plancherel, Hölder's and $a \in L^\infty(\R^2)$, to obtain
\begin{align*}
A & \leq \int_{\sum_{i=1}^3(\tau_i, \xi_i, \mu_i) = 0} |a(\xi_2,\xi_4)| \prod_{i=1}^3\big| \mathcal{F}[Q_{L_i} P_{N_i} u_i] (\tau_i, \xi_i, \mu_i)\big|\\
& \lesssim  \|\mathcal{F}^{-1} [\big|\mathcal{F}[Q_{L_1} P_{N_1} u_1]\big|] \|_{L^2_{txy}}  \|\mathcal{F}^{-1}[ \big|\mathcal{F}[Q_{L_2} P_{N_2} u_2]\big|] \|_{L^4_{txy}} \|\mathcal{F}^{-1}[ \big|\mathcal{F}[Q_{L_2} P_{N_2} u_2]\big|] \|_{L^4_{txy}}.
\end{align*}

Note, by \eqref{Strichartz1}, with $q = 4$ and $r = 4$, we have
\[ \|U(t)P_N \varphi\|_{L^{4}_t L^4_{xy}} \lesssim N^{-\frac{1}{4}}\|\Tilde{P}_N\varphi\|_{L^2_{xy}},\]
hence by Transference principle, Lemma \ref{Transference}, it follows
\[ \|P_N u\|_{L^{4}_t L^4_{xy}} \lesssim N^{-\frac{1}{4}}\|\Tilde{P}_Nu \|_{X^{0,1/2,1}}.\]
Meanwhile, it is of course true that
\begin{align*}
\adnorm{Q_{L_1}P_{N_1}u_1}_{L^2_t L^2_{xy}} \lesssim L_1^{-\frac{1}{2}}\adnorm{Q_{L_1}P_{N_1}u_1}_{X^{0,1/2,1}}.
\end{align*}
After applying the above, with $\mathcal{F}^{-1} \circ |\cdot| \circ \mathcal{F}$ falling out due to Plancherel, overall we have
\[ A \lesssim L_1^{-\frac{1}{2}}N_1^{-\frac{1}{4}}N_3^{-\frac{1}{4}} \prod_{i=1}^3 \| Q_{L_i} P_{N_i} u\|_{X^{0,1/2,1}} \lesssim N_1^{-\frac{9}{4}}N_3^{-\frac{3}{4}} \prod_{i=1}^3 \| Q_{L_i} P_{N_i} u\|_{X^{0,1/2,1}};\]
again the other cases for the maximum of $L_i$ are completely analogous.
For the KP-I case we need only consider the remaining case of $L_1,L_2,L_3 \ll N_1^4N_3$. The lower dispersion, $3$rd order KP-I equation version was first covered by Ionescu-Kenig-Tataru \cite{ionescu2007global}; for general dispersion, which covers our case, the  proof can be found with Lemma 4.2 of Sanwal-Schippa \cite{sanwal2022low}.
\end{proof}

\section{A-priori Estimates}
We move on to provide \textit{a-priori} estimates for solutions to $5$th order KP equations. The following are estimates for solutions of KP equations with general non-linearity $f$;
\begin{equation} \label{KPGneralNonLinear}
\del_x( \del_t u + \del_x^5 u + \del_x f) \pm \del^2_y u = 0.
\end{equation}
This is such that we can ultimately apply the following estimates to the difference and summation of two solutions of the KP equations. Note that \eqref{KPGneralNonLinear} has the Duhamel formulation
\begin{equation} \label{KPGneralNonLinearDuhamel}
u(t) = U(t)u_0 + \int_0^t U(t-s) \del_x f(s) d s.
\end{equation}

The following will find repeated use, as the $L^2_T L^\infty_{xy}$ norm will appear naturally after the use of Hölder's later, so we need a way to return to $L^2$ based Anisotropic Sobolev norm.
\begin{prop} \label{L2TimeEst}
Let $0 < T < 2$ and $u \in L^\infty_T H^{0+,0}$ be a solution to \eqref{KPGneralNonLinear},
then
\begin{equation} \label{L2TimeEstEq}
\adnorm{u}_{L^{2+}_TL^\infty_{xy}} \lesssim \adnorm{u}_{L^\infty_T H^{0+,0}} + \adnorm{J_x^{0+} f}_{L^{2+}_T L^1_{xy}}
\end{equation}

\end{prop}
\begin{proof}
For $N\geq 1$ begin by splitting up $[0,T]$ into $N$ intervals. Consider $J_N = \{I_j: 1 \leq j \leq N\}$ where $I_j = \big[(j-1)\frac{T}{N}, j\frac{T}{N}\big]$ then $\bigcup I_j = [0,T]$.
Using the Duhamel formulation \eqref{KPGneralNonLinearDuhamel}, splitting up the time interval into $J_N$, applying Strichartz estimates both \eqref{Strichartz1} and \eqref{Strichartz2} and using Hölder's to increase time integral exponent afterward; we obtain
\begin{align*}
\adnorm{P_Nu}_{L^{2+}_T L^{\infty}_{xy}}^{2+} &\lesssim \sum_{j \in J_N}(N^{-\frac{1}{2}+}\adnorm{P_N u}_{L^{\infty}_T L^2_{xy}})^{2+} + \sum_{j \in J_N}\left(N^{-1 +}  \adnorm{P_N f_x}_{L^{2-}_{I_j}L^1_{xy}}\right)^{2+} \\
& \lesssim N(N^{-\frac{1}{2}+}\adnorm{P_N u}_{L^\infty_T L^2_{xy}})^{2+} +  \left(N^{0+}\frac{T^{0+}}{N^{0+}}\adnorm{P_N f}_{L^{2+}_{T}L^1_{xy}} \right)^{2+} \\
& \lesssim (N^{0+}\adnorm{P_N u}_{L^\infty_T L^2_{xy}})^{2+} + T^{0+}\left(N^{0+}\adnorm{P_N f}_{L^{2+}_{T}L^1_{xy}} \right)^{2+}.
\end{align*}
For $N \leq 1$, follow the same as above, but avoid splitting up the time interval $[0,T]$. To avoid too much derivative gain, increase the time integral exponent;
\begin{align*}
\|P_N u \|_{L^{2+}_T L^{\infty}_{xy}} & \lesssim T^{4+}\big(\|U(t)P_N u_0 \|_{L^{4-}_T L^\infty_{xy}} + \bigg\|\int_0^t U(t-t')P_N f_x \bigg\|_{L^{4-}_T L^\infty_{xy}} \big) \\
& \lesssim N^{\frac34+}\|P_N u\|_{L^\infty_T L^2_{xy}} + N^{\frac{1}{4}+}\| P_N f_x \|_{L^{2-}_T L^1_{xy}} \\
& \lesssim \|P_N u\|_{L^\infty_T L^2_{xy}} + \| P_N f \|_{L^{2+}_T L^1_{xy}}.
\end{align*}
Summing over $N$, with the use of Littlewood-Paley theory then leads to the result \eqref{L2TimeEstEq}.
\end{proof}

We now move to prove the proposition which gives the full picture of how we aim to bound terms like \eqref{EnergyIntegrand}. 
Since Proposition \ref{Pro1} gives an estimate both in terms of Besov Bourgain space $X^{0,1/2,1}$ and makes use of the dispersion to provide derivative gain, we are left to reduce back to Anisotropic Sobolev space. To do so, we substitute in Duhamel formulations and make use of the following commonly known linear estimates. Since these estimates are well known, we omit proof. See Proposition 3.2 of \cite{ionescu2007global}. 
\begin{lem}\label{BesovBourgain2}
Let $\eta \in C_0^\infty(\R)$ be a smooth time cutoff. For $\varphi \in L^2_{xy}$ and any $K>0$ it holds 
\begin{equation}\label{linearS}
\| \eta(tK) U(t) \varphi \|_{X^{0,1/2,1}} \lesssim \|\varphi \|_{L^2_{xy}}.
\end{equation}
Furthermore for $ K\geq 1$:
\begin{equation}\label{linearD}
\| \eta(tK) \int_0^t U(t -t') f(t') \, dt' \|_{X^{0,1/2,1}} \lesssim K^{-1/2}  \| f\|_{L^2_{txy}}.
\end{equation}
\end{lem}
The fact that these estimates depend on a local time cut-off can be abused to get better estimates by splitting up the time domain. Specifically we do so dependent on a power of the largest spatial frequency projection; helping balance the derivative loss from both linear propagator and integrand terms present in the Duhamel formulations.
In the following we highlight the fact that $a$ being acceptable is now required beyond just $a \in L^\infty$.
\begin{prop} \label{tritri}
Let $0 < T < 2$ and $N_1 \sim N_2 \gtrsim N_3$, with $N_1 \gtrsim 1$. For $f_i \in L^2([0,T]; L^2(\R^2))$, with $a \in L^\infty(\R^2)$ and $\Tilde{a}$ as defined in \eqref{roleswitching} both acceptable. Given $u_1,u_2,u_3 \in C([0,T]; L^2(\R^2))$ satisfying 
\begin{equation} \label{KPuiNonlinear}
\del_x(\del_tu_i + \del_x^5 u_i + \del_x f_i) \pm \del^2_y u_i = 0 
\end{equation}
on $
(0,T) \times \R^2$ for $i = 1,2,3$. Then 
\begin{multline} \label{trilinear}
\left| \int_0^T \int_{\R^2} P_{N_1}u_1(t)\Lambda_a(P_{N_2}u_2(t), P_{N_3}u_3(t)) dt \right| \\
\hspace{-4cm}\lesssim \prod_{i=1}^2 \left( T^\frac{1}{2}\adnorm{P_{N_i}u_i}_{L^2_T L^2_{xy}} +  T^{-\frac{1}{2}}\adnorm{P_{N_i}f_i}_{L^2_T L^2_{xy}} \right) 
\\ \hspace{2cm}\times \left( \adnorm{P_{N_3}u_3}_{L^\infty_T L^2_{xy}} + T^\frac{1}{2} \bigg(  \frac{N_3}{N}\bigg)^\frac{1}{2}\sup_{\substack{I \subset [0,T] \\ |I| \sim TN^{-1}}}  N_3^\frac{1}{2}\adnorm{ P_{N_3}f_{3}}_{L^2_I L^2_{xy}} \right)N_1^{-\frac{5}{4}}N_3^{-\frac{3}{4}} \\
+ T^\frac{1}{2}N_1^{-1} \prod_{i=1}^2 \bigg( \adnorm{P_{N_i} u_i} _{L^\infty_T L^2_{xy}} + T^\frac{1}{2}\adnorm{J_x^{0+}P_{N_i} f_i}_{L^{\infty-}_{T} L^1_{xy}} \bigg) \adnorm{P_{N_3}u_3}_{L^\infty_T L^2_{xy}}.
\end{multline}
\end{prop}
\begin{proof}
We split up $[0,T]$, depending on the highest frequency, into $N_1$ intervals $I_j$ such that $|I_j| \sim TN_1^{-1}$, with $J_{N_1} =\{ I_j : 1 \leq j \leq N_1 - 1\}$, $I_j = [\frac{T}{N_1}(j - \frac{1}{2}), \frac{T}{N_1}(j+\frac{1}{2}) ]$. Observe the slight difference compared to the way the time interval was split in Proposition \ref{L2TimeEst}. We require smooth time cut-offs as opposed to sharp, as we want to make use of Lemma \ref{BesovBourgain2} to get rid of Besov Bourgain norms. Let $\eta \in C^\infty_0(\R)$ with values in $\R_+$, such that $\supp \eta \subset [ -3/4 , 3/4 ]$ and $\eta \equiv 1$ on $[-1/4, 1/4]$ and $\eta + \eta(\cdot - 1) \equiv 1$ on $[1/4, 3/4]$; by construction
\[ \sum_{k \in \Z} \eta(\cdot - k) \equiv 1.\]
The following notation is used to denote translations and dilations of these smooth time cutoffs; $\eta_j: = \eta(\frac{N_1(\cdot - j)}{T})$ and $\Tilde{\eta}_j = \eta(\frac{4N_1(\cdot - j)}{T})$. It is apparent $\eta_j \Tilde{\eta}_j = \eta_j$, and $\supp \eta_j \subset I_j$. We use these smooth cut-offs to split up the characteristic of the time interval, that is 
\begin{equation} \label{timeSplit}\mathbbm{1}_{[0,T]} = \mathbbm{1}_{[0,T]}\eta_0 + \sum_{k = 1}^{N_1 - 1}\eta_k + \mathbbm{1}_{[0,T]}\eta_{N_1}. 
\end{equation}

A clear distinction arises between the boundaries of the time interval and its interior; it becomes necessary for a sharp cut-off at the boundary to avoid the time interval in the associated norms from growing. Each of these cases come with their own challenge. The boundary terms with sharp cut-offs do not behave well with the $X^{0,1/2,1}$ norm, but comprises of only two terms. The interior, which can be estimated with $X^{0,1/2,1}$ requires summation of $N_1$ terms. The plan will consist of substituting in the Duhamel formulations for \eqref{KPuiNonlinear}, after the splitting \eqref{timeSplit}. To do so it is useful to choose an optimal ``centre'' $c_{i,j}$ for the Duhamel formulation depending on which time cut-off we consider; these centres will satisfy
\begin{equation}
\begin{split}\label{Duhamel}
\eta_j P_{N_i}u_i(t) & = \eta_j U(t - c_{i,j})P_{N_i}u_i(c_{i,j}) + \eta_jP_{N_i}F_{i,j}(t) \\
& = \eta_j U(t - c_{i,j})P_{N_i}u_i(c_{i,j}) + \eta_jP_{N_i}\int_{c_{i,j}}^t U(t-s)\eta_j(s) \del_x f_{i}(s) ds. 
\end{split}
\end{equation}
Observe, we also add $\eta_j$ into the integral term to carry over after we apply \eqref{linearD}.
For each of the intervals $I_j$, we choose $c_{i,j} \in I_j$ such that $\adnorm{P_{N_i}u_i(t)}_{L^2_{xy}}$ attains its minimum on $I_j$. This offers us
\[ |I_j|\adnorm{P_{N_i}u_i(c_{i,j})}_{L^2_{xy}}^2 \leq \int_{I_j} \adnorm{P_{N_i}u_i(t)}_{L^2_{xy}}^2\] 
or written differently
\begin{equation} \label{c_ij}
\adnorm{P_{N_i}u_i(c_{i,j})}_{L^2_{xy}}^2 \leq \frac{1}{|I_j|}\adnorm{P_{N_i}u_i}_{L^2_{I_j}L^2_{xy}}^2 \lesssim \frac{N_1}{T}\adnorm{P_{N_i}u_i}_{L^2_{I_j}L^2_{xy}}^2 .
\end{equation}
We formally carry out the time splitting, into terms $A,B,C$, below, where $A,B$ are terms which consider the boundary, while C consists of the interior.
\begin{align*}
\int_0^T & \int_{\R^2} P_{N_1}u_1(t) \Lambda_a(P_{N_2}u_2(t), P_{N_3}u_3(t)) dt \\ 
& = \Gamma_a(\mathbbm{1}_{[0,T]}\eta_0 P_{N_1}u_1, \Tilde{\eta}_0P_{N_2}u_2 ,\Tilde{\eta}_0P_{N_3}u_3 )+ \Gamma_a(\mathbbm{1}_{[0,T]}\eta_{N_1}P_{N_1}u_1, \Tilde{\eta}_{N_1}P_{N_2}u_2 , \Tilde{\eta}_{N_1}P_{N_3}u_3 ) \\
& + \sum_{j=1}^{N_1 - 1} \Gamma_a(\eta_{j}P_{N_1}u_1, \Tilde{\eta}_{j}P_{N_2}u_2 , \Tilde{\eta}_{j}P_{N_3}u_3 ) = : A + B + C.
\end{align*}
Substituting in Duhamel formulation \eqref{Duhamel} for $u_1$ and $u_2$ around their respective centres and time cut-off:
\begin{align*}
C & = C_1 + C_{2,1} + C_{2,2} + C_3 \\
& := \sum_{j=1}^{N_1-1}\Gamma_a(\eta_j U(t - c_{1,j})P_{N_1}u_1(c_{1,j}), \tilde{\eta}_j U(t - c_{2,j})P_{N_2}u_2(c_{2,j}) , \Tilde{\eta}_{j}P_{N_3}u_3 ) \\
&+ \sum_{j=1}^{N_1-1}\Gamma_a(\eta_j U(t - c_{1,j})P_{N_1}u_1(c_{1,j}), \tilde{\eta}_j P_{N_2}F_{2,j}(t)  , \Tilde{\eta}_{j}P_{N_3}u_3 ) \\
& \hspace{5cm}+ \sum_{j=1}^{N_1-1}\Gamma_a(\eta_jP_{N_1}F_{1,j}(t), \tilde{\eta}_j U(t - c_{2,j})P_{N_2}u_i(c_{2,j}) , \Tilde{\eta}_{j}P_{N_3}u_3 ) \\
& + \sum_{j=1}^{N_1-1}\Gamma_a(\eta_jP_{N_1}F_{1,j}(t), \tilde{\eta}_j P_{N_1}F_{2,j}(t)  , \Tilde{\eta}_{j}P_{N_3}u_3 ).
\end{align*}
This is now the point at which we apply the trilinear estimate from Lemma \ref{Pro1}, the estimates \eqref{linearS} and \eqref{linearD} to reduce from Besov Bourgain space back to $L^2$. All of the terms $C_1, C_{2,1}, C_{2,2}$ and $C_3$ are dealt with in this manner.

$\bf{C_1\hspace{0.2cm}estimate}$:
The following is attained by applying Lemma \ref{Pro1}, and then \eqref{linearS} on both norms where $u_1$ and $u_2$ appear, with the third line following from \eqref{c_ij}, and after a simple Cauchy-Scwharz:
\begin{align*}
C_1 &\lesssim N_1^{-\frac{9}{4}}N_3^{-\frac{3}{4}} \sum_{j = 1}^{N_1 - 1}  \adnorm{\eta_j U(t - c_{1,j})P_{N_1}u_1(c_{1,j})}_{X^{0,1/2,1}} \adnorm{\tilde{\eta}_j U(t - c_{2,j})P_{N_2}u_2(c_{2,j})}_{X^{0,1/2,1}} \\
& \hspace{10cm}\times \adnorm{\Tilde{\eta}_{j}P_{N_3}u_3 }_{X^{0,1/2,1}} \\
& \lesssim N_1^{-\frac{9}{4}}N_3^{-\frac{3}{4}} \sup_{1 \leq j \leq N_1 -1}\{ \adnorm{\Tilde{\eta}_{j}P_{N_3}u_3 }_{X^{0,1/2,1}} \} \sum_{j = 1}^{N_1 - 1}  \adnorm{P_{N_1}u_1(c_{1,j})}_{L^2_{xy}} \adnorm{P_{N_2}u_2(c_{2,j})}_{L^2_{xy}} \\
& \lesssim N_1^{-\frac{9}{4}}N_3^{-\frac{3}{4}}\sup_{1 \leq j \leq N_1 -1} \{ \adnorm{\Tilde{\eta}_{j}P_{N_3}u_3 }_{X^{0,1/2,1}} \} \frac{N_1}{T} \sum_{j = 1}^{N_1 - 1}  \adnorm{P_{N_1}u_1}_{L^2_{I_j}L^2_{xy}} \adnorm{P_{N_2}u_2}_{L^2_{I_j}L^2_{xy}}\\
& \lesssim N_1^{-\frac{5}{4}}N_3^{-\frac{3}{4}}\sup_{1 \leq j \leq N_1 -1} \{ \adnorm{\Tilde{\eta}_{j}P_{N_3}u_3 }_{X^{0,1/2,1}} \}T^{-\frac{1}{2}}\adnorm{P_{N_1}u_1}_{L^2_{T}L^2_{xy}} T^{-\frac{1}{2}}\adnorm{P_{N_2}u_2}_{L^2_{T}L^2_{xy}}.
\end{align*}

$\bf{C_3 \hspace{0.2cm}estimate}$: Similar to $C_1$ we apply Lemma \ref{Pro1} but instead of using \eqref{linearS} and \eqref{c_ij}, we use \eqref{linearD} to imply $ \adnorm{\eta_j P_{N_i} F_{i,j}}_{X^{0,1/2,1}} \lesssim T^\frac{1}{2}N_1^{-\frac{1}{2}} \adnorm{\eta_j P_{N_i}\del_x f_i}_{L^2_{txy}} \lesssim T^\frac{1}{2}N_1^\frac{1}{2}\adnorm{\eta_j P_{N_i} f_i}_{L^2_{txy}}$. In this way,
\begin{align*}
C_3 &\lesssim  N_1^{-\frac{9}{4}}N_3^{-\frac{3}{4}} \sup_{1 \leq j \leq N_1 -1}\{ \adnorm{\Tilde{\eta}_{j}P_{N_3}u_3 }_{X^{0,1/2,1}} \} \sum_{j = 1}^{N_1 - 1}  \adnorm{\eta_j P_{N_1} F_{1,j}}_{X^{0,1/2,1}} \adnorm{\tilde{\eta}_jP_{N_2}F_{2,j}}_{X^{0,1/2,1}} \\
& \lesssim  N_1^{-\frac{5}{4}}N_3^{-\frac{3}{4}} \sup_{1 \leq j \leq N_1 -1}\{ \adnorm{\Tilde{\eta}_{j}P_{N_3}u_3 }_{X^{0,1/2,1}} \} T^\frac{1}{2} \adnorm{\eta_j P_{N_1} f_1}_{L^2_{Txy}} T^\frac{1}{2}\adnorm{\tilde{\eta}_j P_{N_2} f_2}_{L^2_{Txy}}.
\end{align*}

$\bf{C_2 \hspace{0.2cm}estimate}$: A combination of applying the method for $C_1$ and $C_3$ to the respective linear and integrand terms treats $C_2$. Below is a treatment of $C_{2,1}$ with terms $U(t-c_{1,j})u_1(c_{1,j})$ and $F_{2,j}$; of course $C_{2,2}$, the term with $F_{1,j}$ and $U(t-c_{2,j})u_2(c_{2,j})$ instead is completely analogous. 
\begin{align*}
C_{2,1} &\lesssim N_1^{-\frac{9}{4}}N_3^{-\frac{3}{4}}\sup_{1 \leq j \leq N_1 -1}\{ \adnorm{\Tilde{\eta}_{j}P_{N_3}u_3 }_{X^{0,1/2,1}} \}  \sum_{j = 1}^{N_1 - 1}  \adnorm{\eta_j U(t - c_{1,j})P_{N_1}u_1(c_{1,j})}_{X^{0,1/2,1}} \\
& \hspace{10cm} \times \adnorm{\tilde{\eta}_jP_{N_2}F_{2,j}}_{X^{0,1/2,1}} \\
&\lesssim N_1^{-\frac{9}{4}}N_3^{-\frac{3}{4}}\sup_{1 \leq j \leq N_1 -1}\{ \adnorm{\Tilde{\eta}_{j}P_{N_3}u_3 }_{X^{0,1/2,1}} \}  \sum_{j = 1}^{N_1 - 1} N_1^\frac{1}{2}T^{-\frac{1}{2}} \adnorm{P_{N_1}u_1}_{L^2_{I_j}L^2_{xy}}  \\
& \hspace{10cm} \times T^\frac{1}{2}N_1^{-\frac{1}{2}} N_1 \adnorm{\tilde{\eta}_j P_{N_2}f_2}_{L^2_{txy}} \\
& \lesssim N_1^{-\frac{5}{4}}N_3^{-\frac{3}{4}}\sup_{1 \leq j \leq N_1 -1}\{ \adnorm{\Tilde{\eta}_{j}P_{N_3}u_3 }_{X^{0,1/2,1}} \} \adnorm{P_{N_1}u_1}_{L^2_{Txy}}   \adnorm{\tilde{\eta}_j P_{N_2}f_2}_{L^2_{Txy}}.
\end{align*}
It remains to get a suitable bound on $\sup_{1 \leq j \leq N_1 -1}\{ \adnorm{\Tilde{\eta}_{j}P_{N_3}u_3 }_{X^{0,1/2,1}} \}$. By substituting in \eqref{Duhamel} for $u_3$ and applying \eqref{linearS} to deal with the linear propagator, we have
\begin{equation*} \label{Xestimate}
\sup_{1 \leq j \leq N_1 -1}  \adnorm{\Tilde{\eta}_{j}P_{N_3}u_3 }_{X^{0,1/2,1}}  \lesssim \adnorm{P_{N_3}u_3}_{L^\infty_T L^2_{xy}} 
+ \sup_{1 \leq j \leq N_1-1} N_3\left( \frac{T}{N_1} \right)^\frac{1}{2} \adnorm{P_{N_3}f_{3}}_{L^2_{I_j}L^2_{xy}}.
\end{equation*}
Overall we are left with
\begin{align} \label{CContribution}
|C| &\lesssim \prod_{i=1}^2 \left( T^\frac{1}{2}\adnorm{P_{N_i}u_i}_{L^2_T L^2_{xy}} + T^{-\frac{1}{2}}\adnorm{P_{N_i}f_i}_{L^2_T L^2_{xy}} \right) 
\\
& \hspace{3cm}
\times \left( \adnorm{ P_{N_3}u_3}_{L^\infty_T L^2_{xy}} + \left( \frac{T}{N_1} \right)^\frac{1}{2} \sup_{\substack{I \subset [0,T] \\ |I| \sim TN^{-1}}}N_3\adnorm{P_{N_3}f_{3}}_{L^2_I L^2_{xy}} \right) N_1^{-\frac{5}{4}}N_3^{-\frac{3}{4}}. \nonumber
\end{align} 
We now consider how to approach $A,B$. We now fully make use of $a, \Tilde{a}$ being acceptable, with \eqref{acceptable} and \eqref{roleswitching} implying both
\[ |\Gamma_a(u_1,u_2,u_3)| \lesssim \|u_1\|_{L^2_{Txy}} \|u_2\|_{L^\infty_{Txy}} \|u_3\|_{L^2_{Txy}}\]
and
\[ |\Gamma_a(u_1,u_2,u_3)| \lesssim \|u_1\|_{L^\infty_{Txy}} \|u_2\|_{L^2_{Txy}} \|u_3\|_{L^2_{Txy}}.\]
By a multi-linear interpolation, Lemma \ref{Multiinterpolation},  it follows that
\[ |\Gamma_a(u_1,u_2,u_3)| \lesssim \|u_1\|_{L^4_{Txy}}\|u_2\|_{L^4_{Txy}} \|u_3\|_{L^2_{Txy}}.\]
Within our context we hence have
\begin{align*}
|A| &= \left|\int_0^T  \int_{\R^2} \eta_0P_{N_1}u_1 \Lambda_a( \tilde{\eta_0}P_{N_2}u_2, \tilde{\eta}_0 P_{N_3}u_3 ) dt dx dy \right| \\
& \leq \adnorm{\eta_0 P_{N_1}u_1}_{L^4_{Txy}} \adnorm{\eta_0 P_{N_2}u_2}_{L^4_{Txy}}  \adnorm{\tilde{\eta}_0 P_{N_3}u_3}_{L^{2}_T L^2_{xy}}.
\end{align*} 
Again using \eqref{Duhamel} for $u_1, u_2$ and Strichartz estimates with admissible pairs $(4,4)$ and $(2+, \infty)$, one gets
\begin{align} \label{L4Est}
\adnorm{P_{N_1}u_1 }_{L^4_{I_0}L^4_{xy}} &\lesssim \adnorm{P_{N_1} U(t)u_1(0)}_{L^4_{I_0} L^4_{xy}} + \adnorm{  \int_0^t U(t-s)P_{N_1}\del_x f_1 }_{L^4_{I_0} L^4_{xy}} \nonumber\\ 
& \lesssim  N_1^{-\frac{1}{4}}\adnorm{P_{N_1} u_1} _{L^\infty_T L^2_{xy}} + N_1^{\frac{1}{4}+}\adnorm{P_{N_1} f_1}_{L^{2-}_{I_0} L^{1}_{xy}} \nonumber\\
& \lesssim  N_1^{-\frac{1}{4}}\adnorm{P_{N_1} u_1} _{L^\infty_T L^2_{xy}} + |I_0|^{\frac{1}{2}}N_1^{\frac{1}{4}+}\adnorm{P_{N_1} f_1}_{L^{\infty-}_{I_0} L^{1}_{xy}} \nonumber\\
& \lesssim  N_1^{-\frac{1}{4}}\adnorm{P_{N_1} u_1} _{L^{\infty}_T L^2_{xy}} + T^{\frac{1}{2}+}N_1^{-\frac{1}{4}+}\adnorm{P_{N_1} f_1}_{L^{\infty-}_{I_0} L^{1}_{xy}}. 
\end{align} 
Using the same argument for $|B|$; we have
\begin{align*}
|A|, |B| \lesssim N_1^{-\frac{1}{2}} \prod_{i=1}^2 \bigg( \adnorm{P_{N_i} u_i} _{L^\infty_T L^2_{xy}} + T^{\frac{1}{2}+}\adnorm{J_x^{0+}P_{N_i} f_i}_{L^{\infty-}_{I_0} L^{1}_{xy}}  \bigg) \bigg( \frac{T}{N_1} \bigg)^\frac{1}{2}\adnorm{P_{N_3}u_3}_{L^{\infty}_T L^2_{xy}}.
\end{align*}
Combining with \eqref{CContribution} we conclude the desired result.
\end{proof}
The requirement that both of $a$ and $\Tilde{a}$ are acceptable as in \eqref{acceptable} is rather strong. As it happens, the only convolution multiplier we use that does not satisfy this is $a_1$ in Lemma \ref{Commutator}; for this we need a minor alteration to Proposition \ref{tritri}.
\begin{prop} \label{tritriException}
Let $0 < T < 2$ and $N_1 \sim N_2 \gtrsim N_3$, with $N_1 \gtrsim 1$. For $f_i \in L^2([0,T]; L^2(\R^2))$, with $a_1$ as defined in Lemma \ref{Commutator}. Given $u_1,u_2,u_3 \in C([0,T]; L^2(\R^2))$ satisfying 
\[ \del_x(\del_tu_i + \del_x^5 u_i + \del_x f_i) \pm \del^2_y u_i = 0 \]
on $
(0,T) \times \R^2$ for $i = 1,2,3$. Then 
\begin{multline} \label{trilinearEX}
\left| \int_0^T \int_{\R^2} \Lambda_{a_1}(P_{N_3}u_3(t), P_{N_2}u_2(t))P_{N_1}u_1(t) dt \right| \\
\hspace{-4cm}\lesssim \prod_{i=1}^2 \left( T^\frac{1}{2}\adnorm{P_{N_i}u_i}_{L^2_T L^2_{xy}} +  T^{-\frac{1}{2}}\adnorm{P_{N_i}f_i}_{L^2_T L^2_{xy}} \right) 
\\ \hspace{2cm}\times \left( \adnorm{u_3}_{L^\infty_T L^2_{xy}} + T^\frac{1}{2} \bigg(  \frac{N_3}{N}\bigg)^\frac{1}{2}\sup_{\substack{I \subset [0,T] \\ |I| \sim TN^{-1}}}  N_3^\frac{1}{2}\adnorm{ P_{N_3}f_{3}}_{L^2_I L^2_{xy}} \right)N_1^{-\frac{5}{4}}N_3^{-\frac{3}{4}} \\
+ T^\frac{1}{2}N_1^{-\frac{1}{4}}N_3^{-\frac{3}{4}} \prod_{i=2}^3 \bigg( \adnorm{P_{N_i} u_i} _{L^\infty_T L^2_{xy}} + T^\frac{1}{2}\adnorm{J_x^{0+}P_{N_i} f_i}_{L^{\infty-}_{T} L^1_{xy}} \bigg) \adnorm{P_{N_1}u_1}_{L^\infty_T L^2_{xy}}.
\end{multline}
\end{prop}
We bring attention to the fact that in the latter product, the roles of $P_{N_1}u_1$ and $P_{N_3}u_3$ have swapped.
\begin{proof}
Proof is exactly as in the proof of Proposition \ref{tritri} up until estimating components $A,B$. We can no longer rely on $\Tilde{a}_1$ being acceptable in its current form. Upon a further analysis
\[ a_1(\xi_1, \xi_2) : =N_3^{-1}\phi_{N_3}(\xi_1)\tilde \phi_N(\xi_2)[\phi_N(\xi_1+\xi_2)(\xi_1+\xi_2)-\phi_N(\xi_2)\xi_2].\]
So the specific form of $\Tilde{a}_1$ can be found to be
\[ \Tilde{a}_1(\xi_1, \xi_2) = a_1(-\xi_1 - \xi_2, \xi_1) = -N_3^{-1}\phi_{N_3}(\xi_1 + \xi_2)\Tilde{\phi}_N(\xi_1)[\phi_N(\xi_2)\xi_2 + \phi_N(\xi_1)\xi_1],\]
meaning its convolution multiplier has form 
\[ \Lambda_{\Tilde{a}_1}(u,v) = -N_3^{-1}P_{N_3}\big(\Tilde{P}_N u \cdot \del_x P_N v + \del_x \Tilde{P}_N P_N u \cdot v  \big).\]
Through Bernstein estimates
\[ \|\Lambda_{\Tilde{a}_1}(u,v) \|_{L^2_x} \lesssim \frac{N}{N_3}\|u\|_{L^\infty_x}\|v\|_{L^2_x},\]
and though this is not exactly acceptability, it is still sufficient for our usage. Making use of this and acceptability of $a_1$ both the following hold
\[ |A| \lesssim \|\eta_0 P_{N_1} u_1\|_{L^2_{Txy}} \|\Tilde{\eta}_0 P_{N_2}u_2\|_{L^2_{Txy}} \|\Tilde{\eta}_0 P_{N_3}u_3\|_{L^\infty_{Txy}},\]
\[ |A| \lesssim \frac{N_1}{N_3}\|\eta_0 P_{N_1} u_1\|_{L^2_{Txy}} \|\Tilde{\eta}_0 P_{N_2}u_2\|_{L^\infty_{Txy}} \|\Tilde{\eta}_0 P_{N_3}u_3\|_{L^2_{Txy}}.\]
Hence by a multi-linear interpolation, Lemma \ref{Multiinterpolation}, it follows that
\[ |A| \lesssim \bigg(\frac{N_1}{N_3} \bigg)^\frac{1}{2}\|\eta_0 P_{N_1} u_1\|_{L^2_{Txy}} \|\Tilde{\eta}_0 P_{N_2}u_2\|_{L^4_{Txy}} \|\Tilde{\eta}_0 P_{N_3}u_3\|_{L^4_{Txy}}.\]
Applying \eqref{L4Est} for both $L^4$ and Hölder in time for $L^2$, with the same argument for $B$, produces the result.
\end{proof}

\section{Estimate on Differences}
We now consider $u_1$ and $u_2$ as two solutions of the KP-equation \eqref{KPEquation}. Let $w = u_1 - u_2$ and $z = u_1 + u_2$, it then follows that $w$ and $z$ satisfy the following:
\begin{equation} \label{DifferenceEq} 
\del_x( \del_t w + \del_x^5 w + \del_x(wz)) \pm \del^2_y w = 0
\end{equation}
\begin{equation} \label{SumEq}
\del_x( \del_t z + \del_x^5 z + \del_x(u_1^2 + u_2^2)) \pm \del^2_y z = 0.
\end{equation}
\begin{lem} \label{NonlinearityEstimates}
Let $u_1, u_2 \in C_t H^{0+,0}$ satisfying the $5$th order KP equation \eqref{KPEquation}. Then for their difference $w$ and sum $z$, the expressions
\[f_{w} = wz\] 
\[f_z = u_1^2 + u_2^2\]
satisfy the following estimates
\[ \| J_x^{0+}f_w \|_{L^{\infty-}_T L^{1}_{xy}} \lesssim \adnorm{w}_{L^\infty_T H^{0+,0}} \bigg( \sum_{i=1}^2 \adnorm{u_i}_{L^\infty_T H^{0+,0}}\bigg) \]
\[ \| J_x^{0+}f_z \|_{L^{\infty-}_T L^{1}_{xy}} \lesssim  \bigg( \sum_{i=1}^2 \adnorm{u_i}_{L^\infty_T H^{0+,0}}^2\bigg).\]
\iffalse
\[ \| J_x^{0+}f_w \|_{L^{\infty-}_T L^{1}_{xy}} \lesssim \adnorm{w}_{L^\infty_T H^{0+,0}} \bigg( \sum_{i=1}^2 (1+ \adnorm{u_i}_{L^\infty_T H^{0+,0}})^3 \bigg) \]
\[ \| J_x^{0+}f_z \|_{L^{\infty-}_T L^{1}_{xy}} \lesssim  \bigg( \sum_{i=1}^2 \adnorm{u_i}_{L^\infty_T H^{0+,0}}(1+ \adnorm{u_i}_{L^\infty_T H^{0+,0}})^3 \bigg).\]
\fi
\end{lem}
\begin{proof}
Both follow quickly by the Fractional Product Rule, Lemma \ref{fractionalProd}.
\iffalse
\[ \adnorm{J_x^{0+}f_w}_{L^{\infty-}_T L^{1+}_{xy}} \lesssim \adnorm{w}_{L^{\infty}_T H^{0+,0}} \adnorm{z}_{L^{\infty-}_T L^{2+}_{xy}} +\adnorm{w}_{L^{\infty}_T L^{2}_{xy}} \adnorm{J_x^{0+}z}_{L^{\infty-}_T L^{2+}_{xy}} \]
and furthermore by Proposition \ref{L2TimeEst}
\begin{align*} \adnorm{J_x^{0+}P_{N}z}_{L^{\infty-}_T L^{2+}_{xy}} &\lesssim N^{0+} \adnorm{P_N z}_{L^\infty_T L^2_{xy}}^{1-}\adnorm{P_N z}^{0+}_{L^2_T L^\infty_{xy}} \\
& \lesssim (1 + \adnorm{z}_{L^{\infty}_T H^{0+,0}})(1 + \adnorm{z}_{L^\infty_T H^{0+,0}} + \adnorm{J_x^{0+}f_z}_{L^\infty_T L^1_{xy}}) \\
& \lesssim \sum_{i=1}^2 (1+ \adnorm{u_i}_{L^\infty_T H^{0+,0}})^3
\end{align*}
The other estimate follows similarly.
\fi
\end{proof}
We finally use all the work done so far to prove the estimate on differences.
\begin{prop} \label{DifferenceEst} For any $s' > s > 0$, given two solutions $u_1,u_2 \in C_T H^{s,0}$ to the KP type equation with same initial data, their difference $w := u_1 - u_2 \in C_T H^{0+,0}$ satisfies the following
\[ \adnorm{w}_{L^\infty_T H^{s,0}}^2 \lesssim \adnorm{w}_{L^\infty_T H^{s,0}}^2\bigg(\sum_{i=1}^2 \|u_i\|_{L^\infty_T H^{s',0}}\big( 1 + \|u_i\|_{L^\infty_T H^{s',0}} \big)^6 \bigg). \]

\end{prop}
\begin{proof}
For low frequencies, $N \lesssim 1$, using the Duhamel formulation of \eqref{DifferenceEq}, and applying Strichartz estimates \eqref{Strichartz2}
\begin{equation} \label{wStrichartzEst}
\adnorm{P_{N}w}_{L^\infty_T H^{s,0}} \lesssim \adnorm{P_N w(0)}_{L^\infty_T H^{s,0}} + \abrac{N}^{2s}N^{\frac{1}{2}+} \adnorm{ P_{N}(wz)}_{L^{2-}_T L^1_{xy}}.
\end{equation}
We may split up components in the following way:
\begin{equation} \label{componentSplit}
P_N(wz) = P_N(\tilde{P}_{N}w P_{\ll N}z) + P_N( P_{\ll N}w\tilde{P}_{N}z) + \sum_{N \lesssim N_1}P_N(\tilde{P}_{N_1}w P_{N_1}z).
\end{equation}
It then follows that
\begin{align*} 
\adnorm{P_N(wz)}_{L^{2-}_T L^1_{xy}} & \lesssim T^\frac{1}{2}(\| \tilde{P}_N w \|_{L^\infty_T L^2_{xy}} \adnorm{P_{\ll N}z}_{L^\infty_T L^2_{xy}} + \adnorm{P_{\ll N}w}_{L^\infty_T L^2_{xy}} \|\tilde{P}_N z\| _{L^\infty_T L^2_{xy}}) \\
& \hspace{3cm} + T^{\frac{1}{2}+} \sum_{N \lesssim N_1} \|\tilde{P}_{N_1} w\|_{L^\infty_T L^2_{xy}}\adnorm{P_{N_1} z}_{L^\infty_T L^2_{xy}},
\end{align*}
and hence we see from \eqref{wStrichartzEst}
\[ \adnorm{P_{\lesssim 1 }w}_{L^\infty_T H^{s,0}} \lesssim \adnorm{P_{\lesssim 1} w(0)}_{L^\infty_T H^{s,0}} + T^\frac{1}{2}\adnorm{z}_{L^\infty_T L^2_{xy}} \adnorm{w}_{L^\infty_T L^2_{xy}}. \]
For high frequencies, we require the use of a similar energy estimate \eqref{Energy}, which in the context of $w$ becomes
\[ \adnorm{P_{N}w}_{L^\infty_T H^{s,0}}^2 \lesssim \adnorm{P_N w(0)}_{L^\infty_T H^{s,0}}^2  + \abrac{N}^{2s} \left| \int_0^T \int_{\R^2} P_N(w) \del_{x}P_N(wz) dx dy dt\right|.\]
Substituting in \eqref{componentSplit} again, we are left with 3 terms to estimate.
\[ A_1: = \sum_{\substack{0 < N_1 \ll N \\ N \gtrsim 1}} \abrac{N}^{2s} \left| \int_0^T \int_{\R^2} P_N(w) \del_{x}P_N(\tilde{P}_{N}w P_{N_1}z) dx dy dt \right|, \]
\[ A_2: = \sum_{\substack{0 < N_1 \ll N \\ N \gtrsim 1}}\abrac{N}^{2s} \left| \int_0^T \int_{\R^2} P_N(w) \del_x P_N( P_{N_1}w\tilde{P}_{N}z) dx dy dt \right|, \]
\[  A_3 := \sum_{1 \lesssim N \lesssim N_1 } \abrac{N}^{2s}  \left| \int_0^T \int_{\R^2} P_N(w) \del_x P_N(\tilde{P}_{N_1}w P_{N_1}z) dx dy dt \right|.\]
Consider first $A_1$
\begin{align}
A_1 =& \sum_{\substack{0 < N_1 \ll N \\ N \gtrsim 1}} \langle  N\rangle^{2s}\int_0^T  \int_{\R^2} \partial_x P_N(\tilde{P}_{N}w P_{ N_1} z ) P_N w \, dtdxdy \nonumber\\
 =&  \sum_{\substack{0 < N_1 \ll N \\ N \gtrsim 1}}  \langle  N\rangle^{2s}\Bigl( -\frac{1}{2}\int_0^T \int_{\R^2} \partial_x P_{N_1}z  (P_N w)^2\, dtdxdy \nonumber\\
 & \hspace{3cm}+\int_0^T \int_{\R^2} [\partial_x P_N, P_{N_1}z] \tilde{P}_N w \cdot P_N w  dtdxdy\Bigr) \\
 = &  \sum_{\substack{0 < N_1 \ll N \\ N \gtrsim 1}}   \langle  N\rangle^{2s} N_1 \int_0^T \int_{\R^2} \Lambda_{a_1}(P_{N_1}z, \tilde{P}_N w) \cdot P_N w  \, dtdxdy \label{intep} \nonumber\\
 & \hspace{3cm}- \frac{1}{2}\sum_{\substack{0 < N_1 \ll N \\ N \gtrsim 1}}  \langle  N\rangle^{2s} N_1 \int_0^T \int_{\R^2} \Lambda_{a_2}(P_{N_1}z, P_N w) \cdot P_N w  \, dtdxdy, \nonumber
\end{align}
where 
\[  a_1(\xi_1,\xi_2)=N_1^{-1}\phi_{N_1}(\xi_1)\tilde \phi_N(\xi_2)[\phi_N(\xi_1+\xi_2)(\xi_1+\xi_2)-\phi_N(\xi_2)\xi_2],\]
is our known commutator from Lemma \ref{Commutator}, and
\[ a_2(\xi_1, \xi_2) = \frac{\xi_1}{N_1}\tilde{\phi}_{N_1}(\xi_1).\]
Clearly $a_2$ and $\Tilde{a}_2$ are acceptable with respect to \eqref{acceptable}. Hence for $A_1$ it is sufficient to estimate a term of the form
\[ \sum_{\substack{0 < N_1 \ll N \\ N \gtrsim 1}} \abrac{N}^{2s} N_1 \int_0^T \int_{\R^2} \Lambda_a( P_{N_1}z, \tilde{P}_N w ) \cdot P_N w dt dx dy\]
for either both $a$ and $\Tilde{a}$ acceptable or just $a = a_1$.
Indeed by Proposition \ref{tritri} and \ref{tritriException} it follows that
\begin{align}
|A_1| &\lesssim \sum_{\substack{0 < N_1 \ll N \\ N \gtrsim 1}}\abrac{N}^{2s}N_1 \bigg(  \left( T\adnorm{P_{N}w}_{L^2_T L^2_{xy}}^2 +  T^{-1}\adnorm{P_{N}f_w}_{L^2_T L^2_{xy}}^2 \right) 
\nonumber \\
&\hspace{2cm}\times \left( \adnorm{P_{N_1}z}_{L^\infty_T L^2_{xy}} + T^\frac{1}{2} \bigg(  \frac{N_1}{N}\bigg)^\frac{1}{2}\sup_{\substack{I \subset [0,T] \\ |I| \sim TN^{-1}}}  N_1^\frac{1}{2}\adnorm{ P_{N_1}f_{z}}_{L^2_I L^2_{xy}} \right)N^{-\frac{5}{4}}N_1^{-\frac{3}{4}} \label{A1fwfzInterior}\\
& \hspace{2cm}+ T^\frac{1}{2}\bigg(\frac{1}{N}\bigg) \bigg( \adnorm{P_{N} w}^2_{L^\infty_T L^2_{xy}} + T^\frac{1}{2}\adnorm{J_x^{0+}P_{N} f_w}^2_{L^{\infty-}_{T} L^{1}_{xy}} \bigg) \adnorm{P_{N_1}z}_{L^\infty_T L^2_{xy}} \label{A1fwfwBoundary} \\
&\hspace{2cm}+T^\frac{1}{2}N_1^{-\frac{3}{4}}N^{-\frac{1}{4}} \bigg( \adnorm{P_{N} w}_{L^\infty_T L^2_{xy}} + T^\frac{1}{2}\adnorm{J_x^{0+}P_{N} f_w}_{L^{\infty-}_{T} L^{1}_{xy}} \bigg) \label{A1fwfzBoundary} \\
& \hspace{4cm}\times\bigg( \adnorm{P_{N_1} z}_{L^\infty_T L^2_{xy}} + T^\frac{1}{2}\adnorm{J_x^{0+}P_{N_1} f_z}_{L^{\infty-}_{T} L^{1}_{xy}} \bigg)\|P_N w\|_{L^\infty_T L^2_{xy}} \bigg).\nonumber
\end{align}
Dealing with the latter terms first, since $\abrac{N}^{2s}N_1 N^{-1} \vee \abrac{N}^{2s}N_1 N^{-1/4}N_1^{-3/4} \lesssim N_1^{2s}$, it is clear that for $N_1 \ll N$ and by Lemma \ref{NonlinearityEstimates}, one has
\begin{multline*} 
\eqref{A1fwfwBoundary} \lesssim \sum_{\substack{0 < N_1 \ll N \\ N \gtrsim 1}}T^\frac{1}{2}N_1^{2s}\bigg( \adnorm{P_{N} w}^2_{L^\infty_T L^2_{xy}} + T^\frac{1}{2}\adnorm{J_x^{0+}P_{N} f_w}^2_{L^{\infty-}_{T} L^{1}_{xy}} \bigg) \adnorm{P_{N_1}z}_{L^\infty_T L^2_{xy}}  \\
 \lesssim T^\frac{1}{2}\adnorm{w}_{L^\infty_T H^{0+,0}}^2\sum_{i=1}^2\adnorm{u_i}_{L^\infty_T H^{2s,0}}\big(1+\| u_i\|_{L^\infty_T H^{0+,0}}\big)^2.
\end{multline*}
Similarly,
\begin{multline*}
\eqref{A1fwfzBoundary} \lesssim \sum_{\substack{0 < N_1 \ll N \\ N \gtrsim 1}}T^\frac{1}{2}N_1^{2s} \bigg( \adnorm{P_{N} w}_{L^\infty_T L^2_{xy}} + T^\frac{1}{2}\adnorm{J_x^{0+}P_{N} f_w}_{L^{\infty-}_{T} L^{1}_{xy}} \bigg) \\
\times\bigg( \adnorm{P_{N_1} z}_{L^\infty_T L^2_{xy}} + T^\frac{1}{2}\adnorm{J_x^{0+}P_{N_1} f_z}_{L^{\infty-}_{T} L^{1}_{xy}} \bigg) \adnorm{P_{N}w}_{L^\infty_T L^2_{xy}}  \\
 \lesssim T^\frac{1}{2}\adnorm{w}_{L^\infty_T H^{0+,0}}^2\sum_{i=1}^2\adnorm{u_i}_{L^\infty_T H^{2s+,0}}\big(1+\| u_i\|_{L^\infty_T H^{0+,0}}\big)^2.
\end{multline*}
For the former term \eqref{A1fwfzInterior}, from Proposition \ref{L2TimeEst} one can get
\begin{align*} 
T\adnorm{w}_{L^2_T L^2_{xy}}^2 + T^{-1}\adnorm{f_w}_{L^2_T L^2_{xy}}^2 &\lesssim 
\adnorm{w}_{L^\infty_T L^2_{xy}}^2 +  T^{0+}\adnorm{w}_{L^\infty_T L^2_{xy}}^2\adnorm{z}_{L^{2+}_T L^\infty_T}^2 \\
&\lesssim \adnorm{w}^2_{L^\infty_T L^2_{xy}}\bigg(1+ \sum_{i=1}^2 \|u_i\|_{L^\infty_T H^{0+,0}}\Big(1 + \|u_i\|_{L^\infty_T L^2_{xy}}  \Big) \bigg)^2.
\end{align*}
Furthermore since $\abrac{N}^{2s}N_1 N^{-\frac{5}{4}}N_1^{-\frac{3}{4}} \lesssim N^{2s - 1}$ with Proposition \ref{L2TimeEst} again
\begin{multline*} \sum_{\substack{0 < N_1 \ll N \\ N \gtrsim 1}}\bigg( \|P_{N_1}z\|_{L^\infty_T L^2_{xy}} + N_1^\frac{1}{2}\|P_{N_1}f_z \|_{L^2_{T}L^2_{xy}} \bigg)N^{2s - 1}  \lesssim \sum_{i=1}^2 \adnorm{u_i}_{L^\infty_T L^2_{xy}}\big(1 + T^{0+}\adnorm{u_i}_{L^{2+}_T L^\infty_{xy}}\big) \\
\lesssim \sum_{i=1}^2 \adnorm{u_i}_{L^\infty_T H^{0+,0}}\big(1 + \adnorm{u_i}_{L^\infty_T H^{0+,0}}\big)^2.
\end{multline*}
Hence we can get the overall estimate, with $2s = 0+$
\[ |A_1| \lesssim \adnorm{w}_{L^\infty_T H^{0+,0}}^2\bigg( \sum_{i=1}^2 \|u_i\|_{L^\infty_T H^{0+,0}} \big(1 + \adnorm{u_i}_{L^\infty_T H^{0+,0}} \big)^6 \bigg).\]
For $A_2$, we can reduce similarly; indeed for $a_3(\xi_1, \xi_2) = \frac{\xi_1 + \xi_2}{N}\phi_N(\xi_1 + \xi_2)$, which can also be seen to be acceptable with respect to \eqref{acceptable}, it follows that
\[ |A_2| \lesssim \sum_{0< N_1 \ll N} \abrac{N}^{2s} N \left| \int_0^T \int_{\R^2} \Lambda_{a_3}(\tilde{P}_Nz, P_{N_1}w) \cdot P_Nw dx dy dt\right|.\]
Again applying Proposition \ref{tritri}
\begin{align}
|A_2| &\lesssim \sum_{\substack{0 < N_1 \ll N \\ N \gtrsim 1}}\abrac{N}^{2s}N  \left( T\adnorm{P_{N}w}_{L^2_T L^2_{xy}} +  T^{-1}\adnorm{P_{N}f_w}_{L^2_T L^2_{xy}} \right) \nonumber\\
&\hspace{2cm}\times \left(T\adnorm{P_{N}z}_{L^2_T L^2_{xy}} +  T^{-1}\adnorm{P_{N}f_z}_{L^2_T L^2_{xy}} \right) \label{A2former}\\
&\hspace{3cm}\times \left( \adnorm{P_{N_1}w}_{L^\infty_T L^2_{xy}} + T^\frac{1}{2} \bigg(  \frac{N_1}{N}\bigg)^\frac{1}{2}\sup_{\substack{I \subset [0,T] \\ |I| \sim TN^{-1}}}  N_1^\frac{1}{2}\adnorm{ P_{N_1}f_{w}}_{L^2_I L^2_{xy}} \right)N^{-\frac{5}{4}}N_1^{-\frac{3}{4}} \nonumber\\
& +\sum_{\substack{0 < N_1 \ll N \\ N \gtrsim 1}} \abrac{N}^{2s}NT^\frac{1}{2}\bigg(\frac{1}{N}\bigg) \bigg( \adnorm{P_{N} w}_{L^\infty_T L^2_{xy}} + T^\frac{1}{2}\adnorm{J_x^{0+}P_{N} f_w}_{L^{\infty-}_{T} L^{1}_{xy}} \bigg) \nonumber \\
& \hspace{3cm} \times\bigg( \adnorm{P_{N} z}_{L^\infty_T L^2_{xy}} + T^\frac{1}{2}\adnorm{J_x^{0+}P_{N} f_z}_{L^{\infty-}_{T} L^{1}_{xy}} \bigg) \adnorm{P_{N_1}w}_{L^\infty_T L^2_{xy}} . \label{A2latter}
\end{align}
Like earlier, dealing with the latter term first, it is clear we can bound using Lemma \ref{NonlinearityEstimates};
\begin{align*}
\eqref{A2latter}& \lesssim \adnorm{w}_{L^\infty_T H^{0+,0}}^2 \bigg(\sum_{i=1}^2 \big( 1 + \|u_i\|_{L^\infty_T H^{0+,0}} \big) \bigg) \times \bigg( \adnorm{z}_{L^\infty_T H^{2s,0}} + \|J_x^{2s+} f_z \|_{L^{\infty-}_T L^{1}_{xy}} \bigg) \\
& \lesssim \adnorm{w}_{L^\infty_T H^{0+,0}}^2\bigg(\sum_{i=1}^2 \|u_i\|_{L^\infty_T H^{2s+,0}} \big( 1 + \|u_i\|_{L^\infty_T H^{0+,0}} \big)^2 \bigg) .
\end{align*}
For the former term \eqref{A2former}, since we can cancel coefficients $\abrac{N}^{2s}NN^{-\frac{5}{4}} \lesssim 1$ bounds follow with another use of Lemmas \ref{NonlinearityEstimates}  and \ref{L2TimeEst}. First, note that
\begin{align*}
& \abrac{N}^{2s}NN^{-\frac{5}{4}}\bigg(T \|P_N w\|_{L^2_T L^2_{xy}} + T^{-1}\| P_N f_w \|_{L^2_T L^2_{xy}} \bigg) \times \bigg(T \|P_N z\|_{L^2_T L^2_{xy}} + T^{-1}\| P_N f_z \|_{L^2_T L^2_{xy}} \bigg) \\
& \hspace{2cm} \lesssim \| w \|_{L^\infty_T L^{2}_{xy}} (1 + \| z \|_{L^{2+}_T L^\infty_{xy}} ) \times \bigg(\sum_{i=1}^2 \|u_i\|_{L^\infty_T L^2_{xy}} (1 + \|u_i \|_{L^{2+}_T L^\infty_{xy}} ) \bigg) \\
& \hspace{2cm} \lesssim \| w \|_{L^\infty_T L^{2}_{xy}} \bigg( \sum_{i=1}^2 \|u_i \|_{L^\infty_T L^2_{xy}}(1 + \|u_i\|_{L^\infty_T H^{0+,0}})^4 \bigg).
\end{align*}
Then for $N_1 \geq 1$, those same lemmas imply
\begin{align*} 
&\sum_{0< N_1 \ll N}N_1^{-\frac{3}{4}}\left( \adnorm{P_{N_1}w}_{L^\infty_T L^2_{xy}} + T^\frac{1}{2} \bigg(  \frac{N_1}{N}\bigg)^\frac{1}{2}\sup_{\substack{I \subset [0,T] \\ |I| \sim TN^{-1}}}  N_1^\frac{1}{2}\adnorm{ P_{N_1}f_{w}}_{L^2_I L^2_{xy}} \right) \\
& \hspace{3cm} \lesssim \|w\|_{L^\infty_T L^2_{xy}}(1 + \|z \|_{L^{2+}_T L^\infty_{xy}}) \\
& \hspace{3cm} \lesssim \|w\|_{L^\infty_T L^2_{xy}} \bigg( \sum_{i=1}^2 (1 + \|u_i\|_{L^\infty_T H^{0+,0}})^2 \bigg).
\end{align*}
For $N_1 \leq 1$, the only term that differs is $N_1^{-\frac{3}{4}}\|P_{N_1}w\|_{L^\infty_T L^2_{xy}}$ since $N_1^{-\frac{3}{4}}$ is not summable. However since $w(0) = 0$, using a similar Strichartz estimate \eqref{Strichartz2} as is done in $\eqref{wStrichartzEst}$
\[ N_1^{-\frac{3}{4}}\|P_{N_1} w \|_{L^\infty_T L^2_{xy}} \lesssim  N_1^{-\frac{3}{4}} N_1 N_1^{\frac{3}{4}-}\|P_{N_1}(wz) \|_{L^{4/3+}_T L^1_{xy}} \lesssim N_1^{1-}\|w\|_{L^\infty_T L^2_{xy}} \|z\|_{L^\infty_T L^2_{xy}},\]
which is summable, and can lead to the same as above.
With the same choice of $s$ as before, we have overall
\[ |A_2| \lesssim \adnorm{w}_{L^\infty_T H^{0+,0}}^2\bigg( \sum_{i=1}^2 \|u_i\|_{L^\infty_T H^{0+,0}} \big(1 + \adnorm{u_i}_{L^\infty_T H^{0+,0}} \big)^6\bigg). \]
Finally for $A_3$ we can estimate using the same $a_3$ as used for $A_2$;
\[ |A_3| \lesssim \sum_{1 \lesssim N \lesssim N_1}\abrac{N}^{2s}N \left| \int_0^T \int_{\R^2} \Lambda_{a_3}(P_{N_1}w, P_{N_1}z) \cdot P_N w dt dx dy \right|.\]
With Proposition \ref{tritri} once more, we have
\begin{align*}
|A_3| & \lesssim \sum_{1 \lesssim N \lesssim N_1} \abrac{N}^{2s}N  \Big( T\|P_{N_1}w\|_{L^2_T L^2_{xy}} + T^{-1}\|P_{N_1}f_w \|_{L^2_T L^2_{xy}} \Big) \\
& \hspace{2cm}\times \Big( T \|P_{N_1}z\|_{L^2_T L^2_{xy}} + T^{-1}\|P_{N_1}f_z\|_{L^2_T L^2_{xy}} \Big) \\
& \hspace{3cm}\times \left( \|P_N w \|_{L^\infty_T L^2_{xy}} +T^\frac{1}{2}\bigg(\frac{N}{N_1}\bigg)^\frac{1}{2}\sup_{\substack{I \subset[0,T] \\ |I| \sim TN_1^{-1}}} N^\frac{1}{2} \|P_N f_w \|_{L^2_I L^2_{xy}} \right)N_1^{-\frac{5}{4}}N^{-\frac{3}{4}} \\
& + \sum_{1 \lesssim N \lesssim N_1}T^\frac{1}{2}\bigg( \frac{1}{N_1} \bigg) \Big( \|P_{N_1}w \|_{L^\infty_T L^2_{xy}} + T^\frac{1}{2}\|J_x^{0+}P_{N_1} f_w \|_{L^{\infty-}_T L^{1+}_{xy}} \Big) \\
& \hspace{3cm}\times \Big( \|P_{N_1}z \|_{L^\infty_T L^2_{xy}} + T^\frac{1}{2}\|J_x^{0+}P_{N_1}f_z \|_{L^{\infty-}_T L^{1+}_{xy}} \Big)\|P_N w \|_{L^\infty_T L^2_{xy}} .
\end{align*}
Observe that this is better than the estimate $\eqref{A2former}+\eqref{A2latter}$ for $A_2$, with the roles of $N$ and $N_1$ swapped; meaning by the exact same argument we obtain
\[  |A_3| \lesssim \adnorm{w}_{L^\infty_T H^{0+,0}}^2\bigg( \sum_{i=1}^2 \|u_i\|_{L^\infty_T H^{0+,0}} \big(1 + \adnorm{u_i}_{L^\infty_T H^{0+,0}} \big)^6\bigg). \]
All three components are thus covered, and the result shown.
\end{proof}

\begin{proof}[Proof of Theorem \ref{mainTheorem}]

Suppose we are given two solutions $u_1,u_2 \in C_T H^{s,0}$ to the $5$th Order KP equations \eqref{KPEquation} with same initial data $u_0$, for some time of existence $T > 0$. We shall conclude unconditional uniqueness via a standard dilation argument. First observe that for $\lambda > 0$, $u_{i,\lambda}(t,x,y) := \lambda^4u_i(\lambda^5t , \lambda x, \lambda^3y) \in C([0, T\lambda^{-5}]; H^{0+,0})$ is also a solution to the $5$th order KP equation \eqref{KPEquation}, now with initial data $u_{0,\lambda} = \lambda^4u_0(\lambda x, \lambda^3y)$. The norm of this dilation can be bounded as
\[ \|u_{i,\lambda}(t)\|_{H^{0+,0}} \lesssim \lambda^2 \|u_i(\lambda^5t) \|_{H^{0+,0}}. \] 
Therefore, for any $0 < \varepsilon \ll 1$, taking
\[ \lambda_\varepsilon = \varepsilon^\frac{1}{2}(1 + \max_{i = 1,2}\|u_i\|_{L^\infty_T H^{0+,0}})^{-\frac{1}{2}}\]
leads to $\|u_{i, \lambda_\varepsilon}\|_{L^\infty_{T_\varepsilon} H^{0+,0}} \lesssim \varepsilon$ with a new time of existence
\[T_\varepsilon = \lambda_\varepsilon^{-5}T = \varepsilon^{-\frac{5}{2}}(1 + \max_{i=1,2}\|u_i\|_{L^\infty_T H^{0+,0}})^{\frac{5}{2}} T.\]
  It is important to note that as we take $\varepsilon$ smaller, the time of existence $T_\varepsilon$ grows unbounded; in particular, for $\varepsilon < T^{\frac{2}{5}}$, $T_{\varepsilon} > 1$. We may therefore assume (for $\varepsilon$ sufficiently small) that $u_{i, \lambda_\varepsilon}$ exists on at least the time interval $[0,1]$. Applying then Proposition \ref{DifferenceEst} to $w_{\varepsilon} := u_{1,\lambda_{\varepsilon}} - u_{2, \lambda_\varepsilon}$ gives us
\[ \|w_\varepsilon \|_{L^\infty([0,1]; H^{0+,0})}^2 \leq C\|w_\varepsilon\|_{L^\infty([0,1]; H^{0+,0})}^2 \varepsilon(1 + \varepsilon)^6. \] 
Therefore, for $\varepsilon$ sufficiently small, that is, such that $C\varepsilon(1 + \varepsilon)^6 < 1$ and $\varepsilon < T^\frac{2}{5}$, we conclude $u_{1, \lambda_\varepsilon} \equiv u_{2, \lambda_\varepsilon}$ on $[0,1]$. Reversing the dilation this gives $u_1 \equiv u_2$ on $[0, \varepsilon^{\frac{5}{2}}(1 + \max_{i=1,2}\|u_i\|_{L^\infty_T H^{0+,0}})^{-\frac{5}{2}} ]$, which can then be extended to $[0,T]$ by iterating the argument.
\end{proof}

\begin{ackno}
The author would like to thank their supervisor Yuzhao Wang for suggesting the problem, and Mahendra Prasad Panthee for proof reading and useful comments.
Supported by EPSRC Mathematical Sciences Small Grant.
\end{ackno}

\appendix
\tocless\section{Proof of Lemma \ref{resonance}}
\addcontentsline{toc}{section}{Appendix}
\begin{proof}
See Lemma 3.4 in Hadac \cite{hadac2008well} for more generality of exponent, but we include our case here for completeness.
Suppose that $|\xi_\text{min}| = |\xi_1|$; note that we will always have $|\xi_\text{min}| \leq \frac{1}{2}|\xi_\text{max}|$ and $|\xi_\text{med}| \geq \frac{1}{2}|\xi_\text{max}|$. Then
\[ |\omega(\xi_1)| = |\xi_\text{min}|^{5} \leq \frac{1}{2^4}|\xi_\text{min}||\xi_\text{max}|^4,\]
furthermore, by the mean value theorem, for some $\theta \in [0,1]$ 
\[|\omega(\xi) - \omega(\xi - \xi_1)| = |\omega'(\xi - \theta\xi_1)||\xi_1| = 5|\xi - \theta\xi_1|^4|\xi_\text{min}|.\]
Since $|\xi_1| \leq |\xi|$
\[ \min_{\theta \in [0,1]}|\xi - \theta\xi_1| = \min\{|\xi|, |\xi-\xi_1|\} = |\xi_\text{med}| \geq \frac{1}{2}|\xi_\text{max}|\]
\[ \max_{\theta \in [0,1]}|\xi - \theta\xi_1| = \max\{|\xi|, |\xi - \xi_1|\} = |\xi_\text{max}|.\]
Hence
\begin{align*}
|\Omega(\xi,\xi_1)| &\geq |\omega(\xi) - \omega(\xi - \xi_1)| - |\omega(\xi_1)| \\
& \geq 5\frac{1}{2^4}|\xi_\text{max}|^4|\xi_\text{min}| - \frac{1}{2^4}|\xi_\text{max}|^4|\xi_\text{min}| = \frac{4}{2^4}|\xi_\text{max}|^4|\xi_\text{min}|;
\end{align*}
\begin{align*}
|\Omega(\xi,\xi_1)| &\leq |\omega(\xi) - \omega(\xi - \xi_1)| + |\omega(\xi_1)| \\
& \leq 5|\xi_\text{max}|^4|\xi_\text{min}| + \frac{1}{2^4}|\xi_\text{max}|^4|\xi_\text{min}| = (5 +\frac{1}{2^4})|\xi_\text{max}|^4|\xi_\text{min}|;
\end{align*}
Which completes the proof in the case $|\xi_1| = |\xi_\text{min}|$. That the cases $|\xi - \xi_1|$ or $|\xi|$ are equal to $|\xi_\text{min}|$ follow from $\Omega(\xi,\xi_1) = \Omega(\xi,\xi - \xi_1) = - \Omega(\xi-\xi_1, \xi_1)$ and repeating the argument above.
\end{proof}
\tocless\section{Proof of Lemma \ref{Commutator}}
\begin{proof}
By the mean value theorem, 
$$
\|a_1(\xi_1,\xi_2)\|_{L^\infty(\R^2}\lesssim N_3^{-1} \phi_{N_3}(\xi_1) |\xi_1| \sup_{\R} \Bigl(\phi_N+|\phi'(\frac{\cdot}{N})(\frac{\cdot}{N})| \Bigr)\lesssim 1 \; ,
$$
hence the second part is shown. For the first we show by writing the commutator with $a:= P_{N_3}g, b:= \tilde{P}_N h$; then
\[ N_3\Lambda_{a_1}(g,h) = [\partial_x P_N, P_{N_3}g]\tilde{P}_Nh = \partial_x P_N( P_{N_3}g \tilde{P}_N h) - P_{N_3}g\partial_x P_N \tilde{P}_N h = \partial_x P_N(ab) - a\partial_x P_N b \]
which can be written as integration of $b$ against kernel $K(x,y) := N^2(a(x) - a(y))\Phi(N(x-y))$ where $\Phi$ is some Scwharz function.
More specifically, this Schwarz function is given by the frequency projection function $\phi$ in the following way:
\[\Phi_N(x) :=  N^2 \Phi(Nx) := \mathcal{F}_x^{-1}[i \xi\phi(\frac{\xi}{N})] = \partial_x \widecheck{\phi}_N = N^2 \partial_x \widecheck{\phi}(Nx).\]
Once we consider the expressions
\[ \partial_x P_N(ab) = \int_\R i\xi \phi(\frac{\xi}{N}) \mathcal{F}_x[ab](\xi) e^{ix\xi} d\xi = \Phi_N * ab (x) = \int_\R a(y)b(y)\Phi_N(x-y) dy\]
\[ a\partial_xP_Nb = a(x) \int_\R i\xi \phi(\frac{\xi}{N})\hat{b}(\xi) e^{ix\xi}d\xi = a(x) \Phi_N * b(x) = \int_\R a(x)b(y)\Phi_N(x-y) dy,\]
following from the correspondence between multiplication and convolution between Fourier transform, the form of the Kernel given before becomes clear.
So indeed the application of the commutator is equal to
\[ \int_{\R}K(x,y)b(y) dy.\]
Since, by mean value theorem,
\[ \sup_{x}\int_{\R}|K(x,y)| dy + \sup_{y}\int_{\R}|K(x,y)| dx \lesssim \|\partial_x a\|_{L^{\infty}},\]
it follows by Schur's test that
\[ \adnorm{ \int_\R K(x,y)b(y) dy }_{L^2_x} \lesssim \|\partial_x a\|_{L^\infty_x} \|b \|_{L^2_x} .\]
Substituting in $a,b$ and using Bernstein estimates on the $L^\infty$ norm cancels the remaining $N_3^{-1}$, yielding the desired result.
\end{proof}
\bibliographystyle{siam}
\bibliography{bibliography}
\end{document}